\documentclass[11pt, a4paper]{amsart}%
\usepackage{amssymb}
\usepackage{amsmath}
\usepackage{amsthm}
\usepackage{amscd,bm}
\usepackage{enumitem,graphicx,tikz, pst-node, pst-plot, pstricks, esint}
\usepackage{geometry}%
\usepackage{hyperref}
\usepackage{xcolor}
\usepackage{cite}
\usetikzlibrary{patterns,shapes,decorations.pathreplacing}

\DeclareMathOperator{\supp}{supp}
\DeclareMathOperator{\loc}{loc}
\DeclareMathOperator*{\diam}{diam}
\DeclareMathOperator{\ind}{ind}

\renewcommand{\Re}{\operatorname{Re}}

\theoremstyle{plain}
\newtheorem{theorem}{Theorem}[section]
\newtheorem{proposition}[theorem]{Proposition}
\newtheorem{corollary}[theorem]{Corollary}
\newtheorem{lemma}[theorem]{Lemma}
\newtheorem{definition}[theorem]{Definition}

\theoremstyle{remark}
\newtheorem{example}[theorem]{Example}
\newtheorem{remark}[theorem]{Remark}

\date{\today}

\begin{document}

\title{Sobolev regularity of occupation measures and paths, variability and compositions}
\author{Michael Hinz}
\address{Universit\"at Bielefeld\\
Fakult\"at f\"ur Mathematik\\
Postfach 100131\\
D-33501 Bielefeld\\
Germany}
\thanks{The research of MH was supported in part by the DFG IRTG 2235 `Searching for the regular in the irregular: Analysis of singular and random systems' and by the DFG CRC 1283, `Taming uncertainty and profiting from randomness and low regularity in analysis, stochastics and their applications'.}
\email{mhinz@math.uni-bielefeld.de}
\author{Jonas M. T\"olle}
\address{Aalto University\\
Department of Mathematics and Systems Analysis\\
PO Box 11100 (Otakaari 1, Espoo)\\
FI-00076 Aalto\\
Finland}
\email{jonas.tolle@aalto.fi}
\thanks{JMT acknowledges support by the Academy of Finland and the European Research Council (ERC) under the European Union's Horizon 2020 research and innovation programme (grant agreements no. 741487 and no. 818437).}
\author{Lauri Viitasaari}
\address{Uppsala University\\
Department of Mathematics\\
751 06\\
Uppsala\\
Sweden}
\email{lauri.viitasaari@math.uu.se}

\begin{abstract}
We prove a result on the fractional Sobolev regularity of composition of paths of low fractional Sobolev regularity with functions of bounded variation. The result relies on the notion of variability, proposed by us in the previous article \cite{HTV20}. Here we work under relaxed hypotheses, formulated in terms of Sobolev norms, and we can allow discontinuous paths, which is new. The result applies to typical realizations of certain Gaussian or L\'evy processes, and we use it to show the existence of Stieltjes type integrals involving compositions. 
\tableofcontents
\end{abstract}


\maketitle

\section{Introduction}

Let $T>0$. To a Borel function $X:[0,T]\to \mathbb{R}^n$ we refer as a \emph{path}. Our main goal in this article is to define compositions $\varphi\circ X$ of paths $X$ and functions $\varphi:\mathbb{R}^n\to \mathbb{R}$, and to provide some information on their regularity. For smooth $X$ and $\varphi$ one can talk about $\varphi\circ X$ in terms of classical calculus, and the principles involved can be generalized in many ways; for instance, $X$ may only be absolutely continuous and $\varphi$ Lipschitz. Here we focus on situations of low regularity, where $X$ may be a `fractal' path  and $\varphi$ is a function of bounded variation (a $BV$-function); examples of paths $X$ we can handle include typical realizations of fractional Brownian motions or certain L\'evy processes. In fact, although our method is analytic, it works particularly well for paths of stochastic processes for which some density information is available. For $BV$-functions (or Sobolev functions) $\varphi$ the definition of a composition $\varphi\circ X$ is no longer trivial: If for instance $n=1$ and $\varphi=\mathbf{1}_{(1,2)}$, then a priori $\varphi$ is not defined at the jump sites $1$ and $2$, and a composition with, say, the path $X=\mathbf{1}_{[1,3)}+\mathbf{1}_{[2,3)}$ has no meaningful definition. 

The usual way out is to use a suitable representative $\widetilde{\varphi}$ of the $BV$- (or Sobolev) class $\varphi$, which is well-defined outside a very small set $N\subset \mathbb{R}^n$ and in some sense `uniquely determined'. If $X$ does not spend positive time in $N$, then $t\mapsto \varphi(X_t)$ provides a correct definition of $\varphi\circ X$ as an element of $L^1(0,T)$. To ensure this we use a condition we refer to as \emph{variability}, \cite[Definition 2.1]{HTV20}. It is a relative condition joint on $X$ and $\varphi$, although we mainly interpret it as a condition on $X$ relative to a fixed $BV$-function $\varphi$. To be more precise, we consider a parametrized family of conditions, namely the finiteness of certain `mutual' generalized energies of the gradient measure of $\varphi$ and the occupation measure of $X$. These conditions provide additional quantitative information on how little time $X$ spends around points at which $\varphi$ has `irregularities'. If we also have some knowledge about the fractional Sobolev regularity of $X$, then variability and this knowledge can be combined to guarantee a certain fractional Sobolev regularity of  $\varphi\circ X$. This means that both the regularity of the path $X$ and its irregularity (encoded in the regularity of its occupation measure) may contribute to the fractional Sobolev regularity of the composition. For the realizations $X$ of prominent stochastic processes results on the a.s. path regularity and the a.s. regularity of occupation measures are well-known, see for instance \cite{CKR93, Garsia, Schilling97, Schilling00} respectively \cite{Ayache, Baraka, Barlow88, BassKhoshnevisan92, Berman69, Berman69a, Berman70, Berman73, Berman85, Bertoin96, Boylan, GemanHorowitz, GetoorKesten, Hawkes, Kahane, Pitt78, Pruitt, SchillingWang, Xiao97, Xiao}, and we combine them to obtain a.s. results for compositions. The regularity of compositions can then for instance be used to ensure the existence of pathwise defined integrals, \cite{Zahle98, Zahle01}, of compositions $\varphi\circ X$ w.r.t. another path $Y:[0,T]\to \mathbb{R}$ of low regularity.

Well-known results on composition operators on Sobolev spaces ask for the boundedness and continuity of $X\mapsto \varphi\circ X$, seen as a nonlinear operator on one and the same Sobolev space or between Sobolev spaces having the same smoothness parameter, and this is possible only for functions  $\varphi$ that are at least locally Lipschitz, \cite{AZ90, MarcusMizel79, RunstSickel96}. Our composition result is different in nature: If $\theta\in (0,1)$ is the order of smoothness of $X$, it claims only that the composition $\varphi\circ X$ is a member of a fractional Sobolev space of a certain lower order $\beta<\theta$ of smoothness, Theorem \ref{thm:sobolev-membership}. 

In \cite{HTV20} we had used variability to define compositions of H\"older paths and $BV$-functions. We had applied this result to ensure the existence of generalized Lebesgue-Stieltjes integrals and to solve certain differential systems with $BV$-coefficients and driven by fractional Brownian paths. Earlier results exploiting the same mechanism can be found in \cite{CLV16, Garzon, Torres, Yaskov} and closely related results in \cite{LeonNualartTindel}. The present article may be seen as a continuation of our results in \cite{HTV20}. One goal is to point out how variability can be discussed in terms of Sobolev regularity of measures: In \cite{HTV20} we had formulated several results under the hypothesis that the measures involved are upper (Ahlfors) regular. Here we use Sobolev norms (of small negative order), whose finiteness may be viewed as integrated upper regularity conditions, \cite{HedbergWolff}, somewhat more flexible than plain upper regularity, see Section \ref{S:energies}. A second goal is to provide a generalization of a key inequality, \cite[Proposition 4.28]{HTV20}, and a regularity result, \cite[Theorem 2.13 (i)]{HTV20}. These results had been formulated under the assumption that the path $X$ is H\"older continuous, as it is the case for typical paths of Gaussian processes, cf. Example \ref{Ex:compofBm}. Here we only require $X$ to be a member of some fractional Sobolev space, see Theorem \ref{thm:sobolev-membership} and Proposition \ref{prop:key-estimate}. This allows to discuss also discontinuous paths $X$, such as typical realizations of L\'evy processes, cf. Example \ref{Ex:compoLevy}. As in \cite{HTV20} we use these results to show the existence of generalized Lebesgue-Stieltjes type integrals, \cite{Zahle98, Zahle01}, see Corollary \ref{cor:composition-integral}. Because the motif is related to our arguments, we add a further section about an inequality proved in \cite{Berman69}, which roughly speaking restricts the possible simultaneous regularity of a (H\"older) path and its occupation measure, Corollary \ref{C:Berman}. A complete and modern discussion of these restrictions can be found in \cite[Theorem 31]{GG20b}. Here we simply wish to point out that the inequality in \cite{Berman69} does not require local times to exist, and that for negative orders of smoothness one can reformulate a condition used in \cite{Berman69} in terms of  packing type measures. 

The class of $BV$-functions is already rich enough to contain functions with discontinuities, singularities or non-Lipschitz oscillations, and there is a well-developed theory on the geometric features of $BV$-functions, \cite{AFP, Ziemer}. 
Of course it is to be expected that the present composition results can be generalized further using a refined Fourier-analytic approach. Important recent results close to the present article can be found in \cite{CatellierGubinelli, GG20a, GG20b, HarangLing, Harang}. In \cite{GG20b} the authors provide a comprehensive study of $\rho$-irregularity from the point of view of prevalence. The notion of $\rho$-irregularity had been introduced in \cite{CatellierGubinelli}, it quantifies the regularity of occupation measures in a rather complete way and has natural consequences for the mapping properties of averaging operators, \cite{CatellierGubinelli, GG20a, Harang}, which are integrals of compositions involving shifted paths. In the present article we do not aim at integrals of compositions but at compositions themselves, clearly a different question. We also point out that $\rho$-irregularity is an `absolute' condition on a given path $X$, while variability in our sense is a condition relative to a fixed function $\varphi$. Certainly $\rho$-irregularity can be a tool to establish variability. See Remarks \ref{R:GGprevalence} and \ref{R:comparetoaverage} for further comments. 

In Section \ref{S:energies} we recall basic notions on potentials and energies. In Sections \ref{S:occu}
and \ref{S:gradient} we collect some observations on the Sobolev regularity of occupation and gradient measures that may typically occur in the situations we are interested in. In Section \ref{S:var} we prove our main result, Theorem \ref{thm:sobolev-membership} on the Sobolev regularity of compositions and provide examples. Section \ref{S:integral} contains an application to the existence of integrals, Corollary \ref{cor:composition-integral}, and Section \ref{S:Berman} the mentioned discussion of a result from \cite{Berman69}, Corollary \ref{C:Berman}.

\section*{Acknowledgements}
We warmly thank Lucio Galeati for kindly bringing the works \cite{GG20a,GG20b} to our attention and the anonymous referees for their masterful advice.

\section{Riesz potentials and energies}\label{S:energies}

We fix some basic notation around Sobolev spaces, energies and potentials. By 
\[\hat{f}(\xi)=\frac{1}{(2\pi)^{n/2}}\int_{\mathbb{R}^n} e^{-ix\xi}f(x)dx,\quad \xi \in\mathbb{R}^n,\]
we denote the Fourier transform $\hat{f}$ of a function $f\in L^1(\mathbb{R}^n)$, and as usual we use the same symbol $\hat{f}$
to denote the Fourier transform of a tempered distribution $f\in\mathcal{S}'(\mathbb{R}^n)$, defined by $\hat{f}(\varphi):=f(\hat{\varphi})$, $\varphi\in\mathcal{S}(\mathbb{R}^n)$, where $\mathcal{S}(\mathbb{R}^n)$ denotes the space of Schwartz functions on $\mathbb{R}^n$. The inverse Fourier transform is denoted by $f\mapsto \check{f}$, interpreted in the respective sense.

For any $0<\gamma<n$ the $(0,+\infty]$-valued function $\xi\mapsto |\xi|^{-\gamma}$ is an element of $\mathcal{S}'(\mathbb{R}^n)$. Its Fourier inverse is called the \emph{Riesz-kernel $k_\gamma$ of order $\gamma$} on $\mathbb{R}^n$. It agrees with the 
$(0,+\infty]$-valued lower semicontinuous function $x\mapsto c(\gamma,n)\:|x|^{\gamma-n}$ on $\mathbb{R}^n$, seen as an element of $\mathcal{S}'(\mathbb{R}^n)$. Here $c(\gamma,n)$ is a well-known constant, see \cite[Section 1.2.2]{AH96}, \cite[Section 6.1]{Grafakos2} or \cite[Section V.1]{Stein70}. The \emph{Riesz-potential of order $0<\gamma<n$} of a Radon measure $\mu$ on $\mathbb{R}^n$ is defined as
\begin{equation}\label{E:Rieszpot}
U^{\gamma}\mu(x):=\int_{\mathbb{R}^n} k_\gamma(x-y)\mu(dy),\quad x\in\mathbb{R}^n.
\end{equation}
If $\mu(dx)=f(x)dx$ with some nonnegative $L^1_{\loc}(\mathbb{R}^n)$, then this is the Riesz-potential of $f$ in the  function sense, \cite{AH96, Landkof, SKM, Stein70}. Given $1\leq q<+\infty$, $0<\gamma<n$ and a Radon measure $\mu$, we define
\begin{equation}\label{E:energyintegral}
I^\gamma_q(\mu):=\int_{\mathbb{R}^n} (U^{\gamma}\mu)^q\:dx.
\end{equation}
We refer to this quantity as the \emph{$(\gamma,q)$-energy of $\mu$}, see \cite[p. 36]{AH96} or \cite{Turesson}.

If $1<p<+\infty$, $\frac{1}{p}+\frac{1}{q}=1$ and $0<\gamma\leq \frac{n}{p}$ then there is a constant $c>0$, depending only on $n$, $p$ and $\gamma$, such that
\begin{equation}\label{E:Wolffineq}
c^{-1} \int_{\mathbb{R}^n} (U^{\gamma}\mu)^q\:dx\leq \int_{\mathbb{R}^n} W^\mu_{\gamma,p}\:d\mu \leq c \int_{\mathbb{R}^n} (U^{\gamma}\mu)^q\:dx
\end{equation}
for any Radon measure $\mu$ on $\mathbb{R}^n$. Here
\begin{equation}\label{E:Wolffpot}
 W^\mu_{\gamma,p}(x):=\int_0^\infty \left(\frac{\mu(B(x,r))}{r^{n-\gamma p}}\right)^{q-1}\frac{dr}{r},\quad x\in\mathbb{R}^n,
\end{equation}
denotes the \emph{Wolff potential} of $\mu$  of orders $p$ and $\gamma$; the notation $B(x,r)$ stands for the open ball of radius $r>0$ centered at $x\in\mathbb{R}^n$. See \cite[Theorem 4.5.4]{AH96} or \cite[Theorem 1]{HedbergWolff}. 

A Radon measure $\mu$ on $\mathbb{R}^n$ is said to be \emph{upper $d$-regular} if there are constants $c>0$ and $0\leq d\leq n$ such that 
\[\mu(B(x,r))\leq c\:r^d,\quad x\in\mathbb{R}^n,\quad 0<r<1.\] 
Integrability conditions for (\ref{E:Wolffpot}) can be viewed as `integrated' upper regularity conditions. For finite measures they are implied by upper $d$-regularity.

\begin{proposition}\label{P:dregular}
If $1<p<+\infty$, $\frac{1}{p}+\frac{1}{q}=1$, $0<\gamma< \frac{n}{p}$ and $\mu$ is a finite Borel measure which is upper $d$-regular with $n-\gamma p<d\leq n$, then $W^\mu_{\gamma,p}$ is bounded and in particular, $I^\gamma_q(\mu)<+\infty$.
\end{proposition}

\begin{proof}
Under the stated hypotheses
\[\int_0^1 \left(\frac{\mu(B(x,r))}{r^{n-\gamma p}}\right)^{q-1}\frac{dr}{r}\leq \int_0^1r^{(d-n+\gamma p)(q-1)}\frac{dr}{r}<+\infty\]
and 
\[\int_1^\infty \left(\frac{\mu(B(x,r))}{r^{n-\gamma p}}\right)^{q-1}\frac{dr}{r}\leq \mu(\mathbb{R}^n)\int_1^\infty r^{(\gamma p-n)(q-1)}\frac{dr}{r}<+\infty.\]
\end{proof}

\begin{remark} It is well-known that if a Radon measure $\mu$ is upper $d$-regular, then the Hausdorff dimension of its support $\supp \mu$ is at least $d$, \cite{Falconer, Mattila}.  A similarly `rigid' upper bound of type $\hat{\mu}(\xi)\leq c|\xi|^{-\gamma}$, $\xi\in\mathbb{R}^n$, on the Fourier transform $\hat{\mu}$ of a finite Borel measure $\mu$ is used to define the Fourier dimension of measures and sets. See for instance \cite{Ekstrom} or \cite[Definition 61]{GG20b}.
\end{remark}

For a finite Borel measure $\mu$ on $\mathbb{R}^n$ the finiteness of an energy of type (\ref{E:energyintegral}) is an expression of a certain Sobolev regularity. Let $\mathcal{P}(\mathbb{R}^n)$ denote the collection of all polynomials on $\mathbb{R}^n$. Given $\alpha\in\mathbb{R}$ and $1<q<+\infty$, the \emph{homogeneous Sobolev space} $\dot{L}_\alpha^q(\mathbb{R}^n)$ is defined as the space of all equivalence classes $f\in \mathcal{S}'(\mathbb{R}^n)/\mathcal{P}(\mathbb{R}^n)$ for which $(|\xi|^\alpha\hat{f})^\vee$
exists and is a member of $L^q(\mathbb{R}^n)$. Endowed with 
\[\left\|f\right\|_{\dot{L}^q_\alpha(\mathbb{R}^n)}:=\big\|(|\xi|^\alpha\hat{f})^\vee\big\|_{L^q(\mathbb{R}^n)}\]
it becomes a normed space. See \cite[Definition 6.2.5]{Grafakos2}. Any finite Borel measure $\mu$ on $\mathbb{R}^n$ is in $\mathcal{S}'(\mathbb{R}^n)$, and being finite, it cannot have a nonzero polynomial part. 

\begin{corollary}\label{C:finitemeasure} Let $\mu$ be a finite Borel measure on $\mathbb{R}^n$.
\begin{enumerate}
\item[(i)] For $0\leq \alpha<n$ we have $\mu\in \dot{L}_{0}^q(\mathbb{R}^n)$ if and only if $\mu$ is absolutely continuous with a density $\varphi\in L^q(\mathbb{R}^n)$ (if $\alpha=0$) respectively of form $U^\alpha\varphi$ with some $\varphi\in L^p(\mathbb{R}^n)$ (if $0<\alpha<n$). In this case, $\left\|\mu\right\|_{\dot{L}^q_\alpha(\mathbb{R}^n)}=\left\|\varphi\right\|_{L^q(\mathbb{R}^n)}$.
\item[(ii)] For $-n<\alpha<0$ we have $\mu \in \dot{L}_{\alpha}^q(\mathbb{R}^n)$ if and only if $I_q^{-\alpha}(\mu)<+\infty$. In this case, $\left\|\mu\right\|_{\dot{L}_{\alpha}^q(\mathbb{R}^n)}^q=I_q^{-\alpha}(\mu)$.
\end{enumerate}
\end{corollary}

Given $1\leq p<+\infty$, $\gamma>0$ and Radon measures $\mu$ and $\nu$ on $\mathbb{R}^n$, we consider the \emph{$(\gamma,p)$-energy of $\mu$ w.r.t. $\nu$}, defined by  
\begin{equation}\label{E:genmutual}
\int_{\mathbb{R}^n} (U^\gamma \mu)^p d\nu.
\end{equation}
The case $p=1$ corresponds to the \emph{mutual Riesz-energy of order $\gamma$ of $\mu$ and $\nu$}, see for instance \cite[Chapter 1, \S 4]{Landkof}. In the case where $\nu$ equals the $n$-dimensional Lebesgue measure $\mathcal{L}^n$ we recover (\ref{E:energyintegral}).

\begin{proposition}\label{P:convolution}
Suppose that $\gamma_1,\gamma_2>0$, $\gamma_1+\gamma_2<n$ and $\mu$, $\nu$ are Radon measures. Then for any integer $p\geq 1$ we have 
\[\int_{\mathbb{R}^n} (U^{\gamma_1+\gamma_2}\mu)^p d\nu=\int_{\mathbb{R}^n}\cdots \int_{\mathbb{R}^n}\int_{\mathbb{R}^n}\prod_{j=1}^p k_{\gamma_1}(x-z_j)\nu(dx)\prod_{j=1}^p U^{\gamma_2}\mu(z_j)\:dz_1\cdots dz_p.\]
\end{proposition}

\begin{proof}
Fubini and the convolution identity $k_{\gamma_1+\gamma_2}=k_{\gamma_1}\ast k_{\gamma_2}$, \cite[Section 25.2]{SKM} or \cite{Stein70}, yield 
\[(U^{\gamma_1+\gamma_2}\mu(x))^m=\int_{\mathbb{R}^n}\int_{\mathbb{R}^n} (U^{\gamma_1+\gamma_2}\mu(x))^{m-1}k_{\gamma_1}(x-z)k_{\gamma_2}(z-y)\mu(dy) dz\]
for all $m\geq 1$, and iterating this identity, the result follows.
\end{proof}

For later use we finally record a weighted version of (\ref{E:Wolffineq}). Let $1<p<+\infty$ and $\frac{1}{p}+\frac{1}{q}=1$. We use the notation 
\[\fint_E fdx=\frac{1}{\mathcal{L}^n(E)}\int_E f(x)dx\] 
for $E\in\mathcal{B}(\mathbb{R}^n)$ of positive and finite measure and $f\in L^1(E)$. Here the symbol $\mathcal{B}(\mathbb{R}^n)$ denotes the Borel $\sigma$-algebra and $\mathcal{L}^n$, as mentioned before, the $n$-dimensional Lebesgue measure. Given $1<p<+\infty$, a measurable function $w:\mathbb{R}^n\to [0,+\infty]$ is called a \emph{weight of class $A_p$} if 
\begin{equation}\label{E:Muck}
\sup_B \left(\fint_B w(x)dx\right)\left(\fint_B w(x)^{-\frac{q}{p}}dx\right)^{p-1}<+\infty,
\end{equation}
with the supremum taken over all open balls $B\subset \mathbb{R}^n$, see \cite{Muckenhoupt, MuckenhouptWheeden} or \cite{Grafakos2, Turesson}. A weight $w$ is of class $A_p$ if and only if $w^{-\frac{q}{p}}$ is of class $A_q$, \cite[Remark 1.2.4]{Turesson}.
\begin{example}\label{E:polyweight}
If $0<\alpha<nq$ then $x\mapsto |x|^{\alpha-n}$ is a weight of class $A_q$, see \cite[Example 1.2.5]{Turesson}.
\end{example}
If $w\in A_p$ and $0<\gamma<\frac{n}{p}$, then there is a constant $c>1$, depending only on $n$, $p$, $\gamma$ and $w$, such that 
\begin{equation}\label{E:Wolffweighted}
c^{-1}\int_{\mathbb{R}^n}(U^\gamma \mu)^q w^{-\frac{q}{p}}dx\leq n\int_{\mathbb{R}^n}W^\mu_{\gamma,p,w}d\mu\leq c\int_{\mathbb{R}^n}(U^\gamma \mu)^q w^{-\frac{q}{p}}dx,
\end{equation}
where 
\[W^\mu_{\gamma,p,w}(x)=\int_0^\infty\left(\frac{r^{\gamma p}\mu(B(x,r))}{w(B(x,r))}\right)^{q-1}\frac{dr}{r},\quad x\in\mathbb{R}^n,\]
with 
\begin{equation}\label{E:densitynotation}
w(E)=\int_E w(x)dx,\quad E\in \mathcal{B}(\mathbb{R}^n). 
\end{equation}
See \cite[Theorem 3.6.6]{Turesson} or \cite[Theorem 3.2]{Adams86}. Note that for $w\equiv\mathbf{1}$ in the integral in (\ref{E:densitynotation}) we recover (\ref{E:Wolffpot}), up to a multiplicative constant.

\begin{remark}
The original statement in \cite{Adams86} is a bit more general, but (\ref{E:Wolffweighted}) as stated here is easy to formulate and sufficient for our needs.
\end{remark}

\section{Regularity of occupation measures}\label{S:occu}

We briefly introduce occupation measures and their energies and then consider examples. Let $I\subset \mathbb{R}$ be a Borel set and $X:I\to \mathbb{R}^n$ a Borel function. The \emph{occupation measure of $X$ over $I$} is the Borel measure $\mu_X^I$ on $\mathbb{R}^n$ defined by
\begin{equation}\label{E:occumeasure}
\mu_X^I(A):=\mathcal{L}^1(X^{-1}(A)\cap I),\quad A\in \mathcal{B}(\mathbb{R}^n).
\end{equation}
Definition (\ref{E:occumeasure}) is equivalent to the validity of the \emph{occupation time formula}
\begin{equation}\label{E:occutimeformula}  
\int_{\mathbb{R}^n}g(x)\mu_X^I(dx)=\int_I g(X_t)\:dt
\end{equation}
for all bounded Borel functions $g:\mathbb{R}^n\to\mathbb{R}$. If $I$ is bounded, then $\mu_X^I$ is finite. If there is a function $L_X^I\in L^1(\mathbb{R}^n)$ such that 
\[\mu_X^I(dx)=L_X^I(x)\mathcal{L}^n(dx),\] 
then, using terminology in a somewhat loose manner, we refer to $L_X^I$ as \emph{local times of $X$ on $I$}. Allowing Borel subsets $E\in \mathcal{B}(I)$ of $I$ in (\ref{E:occumeasure}) in place of $I$ we can define \emph{occupation measures $\mu_X^E$ relative to subsets $E$} and in the absolutely continuous case also \emph{local times $L_X^E$ relative to $E$}. 

Given an absolutely continuous function $X:[0,T]\to\mathbb{R}$, we denote its weak derivative by $t\mapsto X_t'$.

\begin{example}\label{Ex:Lipschitz}
If $I=[0,T]$ and $X:I\to\mathbb{R}$ is Lipschitz, then for any $A\subset \mathbb{R}$ Borel we have
\[\mu_X^I(A)=\mathcal{L}^1((X^{-1}(A)\cap \{X'= 0\})+\int_A \alpha(y)\:dy\]
with the `conditional density' 
\[\alpha(y)=\sum_{t\in \{X=y\}\cap \{X'\neq 0\}}|X_t'|^{-1}\] 
by \cite[Theorem 3.2.3 (2)]{Federer}. As explained in \cite[(2.3) Theorem (b)]{GemanHorowitz}, $X$ has local times if and only if $\mathcal{L}^1(\{X'=0\})=0$, and in this case $L_X^I=\alpha$ $\mathcal{L}^1$-a.e. on $\mathbb{R}$.
\end{example}

Using Corollary \ref{C:finitemeasure} we obtain the following.
\begin{corollary}\label{C:Soboofoccu}
Let $I$ be bounded and $1<q<+\infty$. 
\begin{enumerate}
\item[(i)] For $0\leq \alpha< n$ we have $\mu_X^I\in \dot{L}^q_\alpha(\mathbb{R}^n)$ if and only if $X$ has local times $L_X^I$ in $L^q(\mathbb{R}^n)$ (if $\alpha=0$) respectively of form $L_X^I=U^{\alpha}\varphi$ with $\varphi \in L^q(\mathbb{R}^n)$ (if $0<\alpha<n$). 
\item[(ii)] For  $-n<\alpha<0$ we have $\mu_X^I\in \dot{L}^q_\alpha(\mathbb{R}^n)$ if any only if the energy $I^{-\alpha}_q(\mu_X^I)$ is finite.
\end{enumerate}
\end{corollary}

\begin{remark}\mbox{}
\begin{enumerate}
\item[(i)] Energies of occupation measures are a traditional tool to obtain lower bounds for Hausdorff dimensions of path images $X(E)$ and related sets, \cite{BlumenthalGetoor61, Kahane, Orey70, Taylor53, Taylor55}. For instance, it is well known that if $0<\gamma<n$, $\mu_X^E(\mathbb{R}^n)>0$ and $I_2^{\gamma/2}(\mu_X^E)<+\infty$, then $\dim_H (X(E))\geq n-\gamma$, see \cite[Chapters 17 and 18]{Kahane}. A more contemporary version of these statements can be found in \cite[Theorem 30 and Section 5]{GG20b}. 
\item[(ii)] In \cite{Geman} Geman considered energy integrals of local times of fields and proved conclusions about their local growth and fluctuation. However, the energies studied there are conceptually different from those in Corollary \ref{C:Soboofoccu} (ii) in the sense that they quantify smoothness with respect to the `time interval' $I$ and not the `state space' $\mathbb{R}^n$. 
\end{enumerate}
\end{remark}

If $X$ has local times in the above sense, then they are automatically integrable. Integrability of higher order or even (essential) boundedness naturally imply Sobolev regularity of the occupation measure. From \cite[Proposition 4.14]{HTV20} and  Proposition \ref{P:dregular} we obtain the following.

\begin{proposition}\label{P:LTandupperreg}
If $I=[0,T]$, $1\leq r\leq +\infty$ and $X$ admits local times $L_X^I\in L^r(\mathbb{R}^n)$, then $\mu_X^I$ is upper $(n-\frac{n}{r})$-regular. In this case $\mu_X^I\in \dot{L}_{-\gamma}^q(\mathbb{R}^n)$, provided that $\frac{n}{r}<\frac{\gamma q}{q-1}<n$.
\end{proposition}

A very simple example is as follows.

\begin{example}\label{Ex:Lipschitzcorr}
If $I=[0,T]$ and $X:I\to\mathbb{R}$ is Lipschitz, $|X'|$ is bounded away from zero $\mathcal{L}^1$-a.e. and $X'$ changes sign only finitely many times on $[0,T]$, then $L_X^I\in L^\infty(\mathbb{R})$. By Proposition \ref{P:LTandupperreg}, the measure $\mu_X^I$ is upper $n$-regular.
\end{example}

Typical realizations of certain stochastic processes provide less trivial examples. 

\begin{example}\label{Ex:LT}\mbox{}
\begin{enumerate}
\item[(i)] For large classes of Gaussian stochastic processes it is well-known that typical paths $X$ admit square integrable local times, \cite{Berman69, Berman69a}, \cite[Sections 21 and 22]{GemanHorowitz}. In many cases they are  H\"older continuous in space, \cite{Berman69a, Berman70, Berman73, Pitt78, Xiao97, Xiao06}, \cite[Sections 26 and 27]{GemanHorowitz}. 
There are classes of Gaussian processes with differentiable or smooth local times, see for instance \cite[Section 28]{GemanHorowitz} or \cite{Harang}. An $n$-dimensional fractional Brownian motion $X:I\to \mathbb{R}^n$ on a bounded interval $I$ with Hurst index $0<H<1$ and over some probability space $(\Omega,\mathcal{F},\mathbb{P})$ has local times $L_X^I$ if and only if $H<\frac{1}{n}$, see for instance \cite{Ayache, Baraka, Berman69a, Berman70, Xiao06} and the references cited there or \cite[Section 30]{GemanHorowitz}. In this case $L_X^I$ is $\mathbb{P}$-a.s. $\gamma'$-H\"older continuous on $\mathbb{R}^n$, provided that 
\[0<\gamma'<1\wedge \Big(\frac{1}{2nH}-\frac{1}{2}\Big),\] 
see \cite[(30.4) Theorem]{GemanHorowitz}.
\item[(ii)] There are classical sufficient conditions for the existence of local times for Markov processes $X$ over some probability space $(\Omega,\mathcal{F},\mathbb{P})$ and with values in $\mathbb{R}$, see \cite[Theorem 1]{GetoorKesten} or \cite[(17.1) Theorem]{GemanHorowitz}. For the special case that $X$ is a L\'evy process with symbol $\psi$ the existence of local times $L_X^I$ $\mathbb{P}$-a.s. in $L^2(\mathbb{R}^n)$ follows if 
\[\int_\mathbb{R}\Re((1+\psi(\xi))^{-1})d\xi<+\infty,\] 
see \cite[Chapter V, Theorem 1]{Bertoin96} or \cite{Hawkes}. The $\mathbb{P}$-a.s. (joint) continuity of local times follows for instance if the diffusion coefficient is strictly positive or the L\'evy measure $\nu$ satisfies $\nu(\mathbb{R}\setminus \{0\})=+\infty$, $0$ is regular for $\{0\}$ in the sense of Markov processes, \cite[Chapter I, Definition 11.1]{BlumenthalGetoor68}, and we have 
\[\sum_{j=1}^\infty \delta(2^{-j})^{\frac{1}{2}}<+\infty,\] 
where 
\[\delta(u)=\sup_{|x|\leq u}\pi^{-1}\int_\mathbb{R} (1-\cos x\xi)\Re((1+\psi(\xi))^{-1})d\xi,\] 
see \cite[Theorem 4]{GetoorKesten}. For other classical results on (joint) continuity, see \cite{Barlow88, BassKhoshnevisan92, Berman85} or \cite[Chapter V, Theorem 15]{Bertoin96}. For a symmetric $\alpha$-stable L\'evy process $X$ taking values in $\mathbb{R}$ we have $\psi(\xi)=c|\xi|^\alpha$, $\xi\in\mathbb{R}$. If $1<\alpha\leq 2$, then it has $\mathbb{P}$-a.s. jointly continuous local times, \cite{Boylan, Trotter}. A result on the existence of square integrable local times for Feller processes can be found in \cite[Theorem 1.2 and Section 3.2]{SchillingWang}.
\item[(iii)]  Integrability properties of local times have also been observed for deterministic self-affine functions, see for instance \cite[p. 438, Remarques]{Bertoin}.
\end{enumerate}
\end{example} 

The existence of local times, which restricts the range of possible space dimensions $n$ particularly heavily, is not needed to observe Sobolev regularity properties of occupation measures.

\begin{example}\label{Ex:occumeas}\mbox{}
\begin{enumerate}
\item[(i)] If $0<d< n$, then standard calculations show that for an $n$-dimensional fractional Brownian motion $X=B^H$ over some probability space $(\Omega,\mathcal{F},\mathbb{P})$ we have 
\[\sup_{x\in\mathbb{R}^n}\mathbb{E}\int_0^T|X_t-x|^{-d}dt<+\infty,\] 
provided that its Hurst index $H$ satisfies $H<\frac{1}{d}$. See for instance \cite[Example 4.22]{HTV20}. For such $d$ it follows that the occupation measure of $X$ on $I=[0,T]$ is $\mathbb{P}$-a.s. upper $d$-regular, see for instance \cite[Proposition 4.12]{HTV20}, and by Proposition \ref{P:dregular} therefore $\mu_X^I\in \dot{L}^q_{-\gamma}(\mathbb{R}^n)$ $\mathbb{P}$-a.s. if $0<n-\gamma p<\frac{1}{H}$ and $p$ and $q$ are conjugate. The situation is consistent with the local time case $H<\frac{1}{n}$ in the sense that $\mathbb{P}$-a.s. boundedness of local times implies the $\mathbb{P}$-a.s. upper $n$-regularity of the occupation measure. For $p=q=2$ one can, alternatively, use the fact that 
\[\mathbb{E}[|X(t+\tau)-X(t)|^{-d}]\leq c\:\tau^{-dH}\] 
and follow another standard calculation to obtain
\[\mathbb{E}\left\|\mu_X^I\right\|_{\dot{L}^2_{-\gamma}(\mathbb{R}^n)}^2\leq c\:\mathbb{E}\int_0^T\int_0^T|X(t)-X(\tau)|^{2\gamma-n}\:dt\:d\tau,\]
which is finite if $n-2\gamma<\frac{1}{H}$, see for instance \cite[Chapter 16]{Falconer}. Here we have used Proposition \ref{P:convolution} (with $p=1$).
\item[(ii)] Suppose that $X$ is a L\'evy process with values in $\mathbb{R}^n$ over a probability space $(\Omega,\mathcal{F},\mathbb{P})$. Its \emph{Pruitt index} $\gamma_0$ is the supremum over all $d\geq 0$ such that 
\[\mathbb{E}\int_0^T |X(t)|^{-d}dt<+\infty,\] 
see \cite[formula (1.1) and Theorem 2]{Pruitt} or \cite[Remark 4.3]{Xiao}. The stationarity of increments allows to obtain 
\begin{multline}
\mathbb{E}\int_0^T\int_0^T|X(t)-X(\tau)|^{2\gamma-n}\:dt\:d\tau\notag\\
=\int_0^T\int_0^T\mathbb{E}[|X(|\tau-t|)|^{2\gamma-n}]\:dt\:d\tau\notag\leq 2T \int_0^T\mathbb{E}[|X(t)|^{2\gamma-n}]\:dt
\end{multline}
by a similar calculation as in (i), and we observe this is finite if $n-2\gamma<\gamma_0$. For such $\gamma$ we then have $\mu_X^I\in \dot{L}^2_{-\gamma}(\mathbb{R}^n)$ $\mathbb{P}$-a.s. In the special case that $X$ is an isotropic $\alpha$-stable  L\'evy process with $0<\alpha\leq 2$, we have $\gamma_0=\alpha\wedge n$, and in the general case we have $\beta'\wedge n\leq \gamma_0\leq \beta$, where the indices $\beta'$ and $\beta$ are defined as in \cite[Definition 2.1 and Section 5]{BlumenthalGetoor61}, see \cite[Theorem 5]{Pruitt}. Indices for Feller processes, including a generalization of $\beta$, have been studied in \cite{Schilling98}.
\end{enumerate}
\end{example}

\begin{remark}\label{R:GGprevalence}
In \cite[Theorem 28]{GG20b} it is shown (in the language of $\rho$-irregularity and prevalence) that the property to have quite regular occupation measures or even smooth or analytic local times, is generic for  H\"older continuous, continuous and $p$-integrable paths. A main ingredient for this theorem is a precise observation of how local nondeterminism properties of Gaussian processes imply their irregularity, \cite[Theorem 29]{GG20b}. Large lists of examples in the classes of Gaussian and stable processes, which naturally include the classical examples mentioned above, may be found in \cite[Sections 4.2 and 4.4 respectively 4.3]{GG20b}. Extreme regularization effects are observed in \cite[Theorem 4]{Harang}. The a.s. space-time regularity of local times of Volterra-L\'evy processes is investigated in \cite[Theorem 1]{HarangLing}.
\end{remark}

\section{Regularity of gradient measures}\label{S:gradient}

We introduce further norms and function spaces needed later on.
Let $U\subset \mathbb{R}^n$ be a domain and $\nu$ a Radon measure on $U$. For $m\geq 1$, $1\leq p<+\infty$ and $\theta\in(0,1)$ we write
\begin{equation}\label{E:Gagliardo}
[f]_{\theta,p;\nu} = \left(\int_U \int_U \frac{|f(x)-f(y)|^p}{|x-y|^{n+\theta p}}\:\nu(dx)\nu(dy)\right)^{\frac{1}{p}}
\end{equation}
for the \emph{$(\theta,p)$-Gagliardo seminorm} of a function $f\in L^1_{\loc}(U,\nu;\mathbb{R}^m)$. We denote the fractional Sobolev space of all $f\in L^p(U,\nu;\mathbb{R}^m)$ such that $\Vert f \Vert_{W^{\theta,p}(U,\nu;\mathbb{R}^m)}: = \Vert f \Vert_{L^p(U,\nu;\mathbb{R}^m)} + [f]_{\theta,p}$ is finite by $W^{\theta,p}(U,\nu;\mathbb{R}^m)$. If $\nu=\mathcal{L}^n$, then we drop it from notation and simply write $[f]_{\theta,p}$ and $W^{\theta,p}(U;\mathbb{R}^m)$. We use the notation
\[[f]_{\theta,\infty}:=\sup_{x,y\in U,\ x\neq y}\frac{|f(x)-f(y)|}{|x-y|^{\theta}}\]
for the \emph{$\theta$-H\"older seminorm}. Strictly speaking, $U$ should be part of the notation on the left-hand side, but its choice will always be clear from the context. For the spaces $C^\theta(U;\mathbb{R}^m)$ of bounded $\theta$-H\"older continuous functions on $U$ we write $W^{\theta,\infty}(U;\mathbb{R}^m):=C^\theta(U;\mathbb{R}^m)$ in order to keep statements of results short. We emphasize that this is not an established standard notation. If $m=1$ we omit $\mathbb{R}=\mathbb{R}^m$ from notation.

By $BV_{\loc}(\mathbb{R}^n)$ we denote the space of \emph{functions locally of bounded variation ($BV_{\loc}$-functions)} on $\mathbb{R}^n$, that is, functions $\varphi\in L^1_{\loc}(\mathbb{R}^n)$ whose distributional partial derivatives $D_i\varphi$ are signed Radon measures, $i=1,...,n$, \cite{AFP,Ziemer}. We write $D\varphi=(D_1\varphi,...,D_n\varphi)$ for the $\mathbb{R}^{n}$-valued gradient measure of a function $\varphi\in BV_{\loc}(\mathbb{R}^n)$, and $\left\|D\varphi\right\|$ for the total variation of $D\varphi$. If $\varphi \in L^1(\mathbb{R}^n)$ and $\left\|D\varphi\right\|(\mathbb{R}^n)<+\infty$, then $\varphi$ is said to be a \emph{function of bounded variation ($BV$-function)}, and we denote the space of such functions by $BV(\mathbb{R}^n)$.

\begin{remark}
Recall that in general a $BV$-function does not have to be an element of a Sobolev space $W^{\theta,p}$. Counterexamples are indicators of sets of finite perimeter in higher dimensions, \cite{AFP}. Functions on bounded domains $U$ which are $\theta'$-H\"older for some $\theta'>\theta$ are in $W^{\theta,p}(U)$. In general, H\"older continuous functions do not have to be $BV$. Counterexamples are Weierstrass functions, \cite{Falconer}.
\end{remark}

If the total variation of the gradient measure $\left\|D\varphi\right\|$ of a $BV$-function $\varphi$ has a bounded Riesz potential (\ref{E:Rieszpot}) or a finite energy (\ref{E:energyintegral}), then this is just an expression of its additional H\"older or Sobolev regularity. We state the following Proposition \ref{P:Riesztosmooth}, along with examples, to briefly point out that --- in contrast to the boundedness condition in (i) --- integrability conditions for potentials of gradient measures as in (ii) do not force the function to be continuous, unless one enters the regime of Sobolev's embedding theorem. Proposition \ref{P:Riesztosmooth} will not be used in later sections.

\begin{proposition}\label{P:Riesztosmooth}
Let $\varphi\in BV_{\loc}(\mathbb{R}^n)$, let $\nu$ be an $n$-upper regular Radon measure on $\mathbb{R}^n$ and $s\in (0,1)$. 
\begin{enumerate}
\item[(i)] If $\sup_{z\in \mathbb{R}^n} U^{1-s}\left\|D\varphi\right\|(z)<+\infty$, then $\varphi$ has a Borel version $\widetilde{\varphi}$ which is H\"older continuous of order $s$ on $\mathbb{R}^n$, and 
\begin{equation}\label{E:Hoelderversion}
[\widetilde{\varphi}]_{s,\infty}\leq c\: \sup_{z\in \mathbb{R}^n} U^{1-s}\left\|D\varphi\right\|(z)
\end{equation}
with a constant $c>0$ depending only on $n$ and $s$.
\item[(ii)] If $1<p<+\infty$, $U^{1-s}\left\|D\varphi\right\| \in L^p(\mathbb{R}^n,\nu)$, then for any $0<\beta<s$ we have 
\begin{equation}\label{E:Gagliardobounded}
[\varphi]_{\beta,p;\nu}^p\leq c\:\int_{\mathbb{R}^n} (U^{1-s}\left\|D\varphi\right\|(x))^p \nu(dx)
\end{equation}
with a constant $c>0$ depending only on $n$, $s$, $\beta$, $p$ and $\nu$. If in addition $\varphi\in L^p(\mathbb{R}^n,\nu)$, then 
$\varphi\in W^{\beta,p}(\mathbb{R}^n,\nu)$. If under these assumptions we have $\nu=\mathcal{L}^n$ and $\frac{n}{p}<\beta$, then $\varphi$ has a Borel version $\widetilde{\varphi}$ which is H\"older of order $\beta-\frac{n}{p}$. 
\end{enumerate}
\end{proposition}

\begin{example}\label{Ex:gradientexamples}\mbox{}
\begin{enumerate}
\item[(i)] Given a point $a$ in space, let $\delta_a$ be the point mass measure assigning measure one to $\{a\}$. For $-\infty<a<b<+\infty$ we clearly have $\varphi=\mathbf{1}_{(a,b)}\in BV(\mathbb{R})$ and $\left\|D\varphi\right\|=\delta_{a}+\delta_{b}$. Moreover, $U^{1-s}\delta_{a}(x)=k_{1-s}(x-a)$, which is in $L^p(\mathbb{R},e^{-|x|}dx)$ if $sp<1$, similarly for $\delta_{a}$. Estimate (\ref{E:Gagliardobounded}) reproduces the well-known fact that $\mathbf{1}_{(a,b)}\in W^{\beta,p}(\mathbb{R},e^{-|x|}dx)$ if $\beta p<1$.
\item[(ii)] The Riesz-kernel $k_\gamma$ itself is in $L^1(\mathbb{R}^n,e^{-|x|}dx)$. It is singular at the origin, but away from the origin it is $C^\infty$. It satisfies $|\partial_{x_i}k_\gamma(x)|\leq c|x|^{\gamma-n-1}$. Consequently $\partial_{x_i}k_\gamma\in L^1(\mathbb{R}^n,e^{-|x|}dx)$ if $1<\gamma<n$, and in particular $k_\gamma\in BV_{\loc}(\mathbb{R}^n)$. It follows that for $n\geq 2$ and $\frac{n}{2}+s<\gamma<n$ we have $U^{1-s}\left\|Dk_\gamma\right\| \in L^2(\mathbb{R}^n,e^{-|x|}dx)$, although $k_\gamma$ is not continuous on $\mathbb{R}^n$.
\end{enumerate}
\end{example}

We recall some notions needed in the proof of Proposition \ref{P:Riesztosmooth}. Let $\varphi\in L^1_{\loc}(\mathbb{R}^n)$. If for $x\in  \mathbb{R}^n$ 
there is some $\lambda_\varphi(x)\in\mathbb{R}$ such that 
\[\lim_{r\to 0}\frac{1}{\mathcal{L}^n(B(x,r))}\int_{B(x,r)}|\varphi(y)-\lambda_{\varphi}(x)| dy=0,\]
then $\lambda_\varphi(x)$ is called the \emph{approximate limit} of $\varphi$ at $x$. We write $S_\varphi$ for the 
set of points $x\in  \mathbb{R}^n$ for which this property does not hold. The set $S_\varphi$ is Borel and of zero Lebesgue measure, \cite[Proposition 3.64]{AFP}, and it does not depend on the choice of a representative for $\varphi$, \cite[p. 160]{AFP}. A point $x\in \mathbb{R}^n\setminus S_\varphi$ with $\varphi(x)=\lambda_\varphi(x)$ is called a \emph{Lebesgue point} of $\varphi\in L^1_{\loc}(\mathbb{R}^n)$. We call a Borel function $\widetilde{\varphi}:\mathbb{R}^n\to\mathbb{R}$ a \emph{Lebesgue representative} of $\varphi\in L^1_{\loc}(\mathbb{R}^n)$ if $\widetilde{\varphi}(x)=\lambda_{\varphi}(x)$ for all $x\in \mathbb{R}^n\setminus S_\varphi$. 

If $\nu$ is a Borel measure on $\mathbb{R}^n$ and $0< \gamma< n$, we write 
\begin{equation}\label{E:fracmaxfct}
\mathcal{M}_{\gamma} \nu(x):=\sup_{r>0}r^{\gamma-n}\:\nu(B(x,r)), \quad x\in \mathbb{R}^n,
\end{equation}
to denote the \emph{fractional Hardy-Littlewood maximal function} of $\nu$ of order $\gamma$. It is immediate that 
\begin{equation}\label{E:trivialbound}
\mathcal{M}_{\gamma}\nu(x)\leq c(\gamma,n)^{-1}\:U^\gamma\nu(x),\quad x\in\mathbb{R}^n.
\end{equation}

\begin{proof}[Proof of Proposition \ref{P:Riesztosmooth}] Throughout the proof let $\varphi$ denote a Lebesgue representative of the $\mathcal{L}^n$-equivalence class denoted by the same symbol. From \cite[Proposition C.1]{HTV20} (which itself is based on \cite[Lemma 4.1]{AK} and \cite[Remark 3.45]{AFP}) we know that for any Lebesgue points $x,y\in\mathbb{R}^n$ of $\varphi$ we have 
\begin{equation}\label{E:backendMVT}
|\varphi(x)-\varphi(y)|\leq c|x-y|^s\left(\mathcal{M}_{1-s}\left\|D\varphi\right\|(x)+\mathcal{M}_{1-s}\left\|D\varphi\right\|(y)\right)
\end{equation}
with $c>0$ depending only on $n$ and $s$. Combining this with the hypothesis in statement (i) and with (\ref{E:trivialbound}), we see that $\varphi$ is H\"older continuous of order $s$ on $\mathbb{R}^n\setminus S_\varphi$. Hence it extends to a H\"older continuous function on $\mathbb{R}^n$, and estimate (\ref{E:Hoelderversion}) also follows along these arguments. To see an elementary proof of (ii), note that combining (\ref{E:backendMVT}) and (\ref{E:trivialbound}), taking the $p$-th power, integrating with respect to $|x-y|^{-(n+\beta p)}\nu(dx)\nu(dy)$ and using the symmetry of the integrand, we obtain
\begin{multline}
\int_{\mathbb{R}^n} \int_{\mathbb{R}^n} \frac{|\varphi(x)-\varphi(y)|^p}{|x-y|^{n+\beta p}}\nu(dx)\nu(dy)\notag\\
\leq c\int_{\mathbb{R}^n} \left(\int_{\mathbb{R}^n} |x-y|^{-n+(s-\beta)p}\nu(dy)\right)\left(U^{1-s}\left\|D\varphi\right\|(x)\right)^p\nu(dx)
\end{multline}
which shows (\ref{E:Gagliardobounded}). The second statement in (ii) is clear from the Sobolev embedding theorem, see for instance \cite{DiNezza}.
\end{proof}

\begin{remark}
Measures $\nu$ different from $\mathcal{L}^n$ in (\ref{E:Gagliardo}) and Proposition \ref{P:Riesztosmooth} were used only to avoid integrability issues at infinite in the examples. Except for a single example (Example \ref{Ex:LTtovar} below) we will always choose $\nu=\mathcal{L}^n$ in what follows.
\end{remark}

\section{Variability and compositions}\label{S:var}

In this section we formulate our main statement, Theorem \ref{thm:sobolev-membership}, on the existence and regularity of compositions of fractional Sobolev paths $X$ with $BV$-functions $\varphi$. It involves regularity of the path in terms of fractional Sobolev spaces and a relative diffusivity condition for the occupation measure of the path $X$ and the gradient measure of $\varphi$, Definition \ref{D:SVp-condition}. In this section we keep the interval $I=[0,T]$ fixed.

Recall that we refer to any Borel function $X=(X^1,\ldots,X^n):[0,T]\to \mathbb{R}^n$ as a \emph{path}. The following notion had been introduced in \cite[Definition 2.1]{HTV20}. 
\begin{definition}\label{D:SVp-condition}
Let $\varphi \in BV(\mathbb{R}^n)$, $p\in [1,+\infty]$ and $s\in (0,1)$. We say that a path $X:[0,T]\to \mathbb{R}^n$ \emph{is $(s,p)$-variable with respect to $\varphi$} if 
\begin{equation}
\label{eq:SVp-condition}
\int_{\overline{X([0,T])}}k_{1-s}(X_{\cdot}-z)\left\|D\varphi\right\|(dz) \in L^p(0,T).
\end{equation} 
We write $V(\varphi,s,p)$ for the class of paths $X$ that are $(s,p)$-variable w.r.t. $\varphi$ and use the short notation $V(\varphi,s):=V(\varphi,s,1)$.
\end{definition}

\begin{remark}\mbox{}
\begin{enumerate}
\item[(i)] If $X([0,T])$ is bounded, then one can define $(s,p)$-variability with respect to $\varphi \in BV_{\loc}(\mathbb{R}^n)$ via \eqref{eq:SVp-condition}.
\item[(ii)] Obviously $V(\varphi,s,p')\subset V(\varphi,s,p)$ for $p'>p$ and $V(\varphi,s',p)\subset V(\varphi,s,p)$ for $s'>s$. 
\item[(iii)] It is immediate from (\ref{E:occutimeformula}) that $X$ is $(s,p)$-variable  w.r.t. $\varphi$ if and only if the $(1-s,p)$-energy (\ref{E:genmutual}) of $\left\|D\varphi\right\|$ w.r.t. $\mu_X^I$ is finite. In the case that local times exist this happens if and only if the Riesz-potential of order $1-s$ of $\left\|D\varphi\right\|$ is in the weighted $L^p$-space on $\mathbb{R}^n$ with weight $L_X^I$. 
\end{enumerate}
\end{remark}

\begin{example}\label{Ex:indicator}
If $n=1$ and $\varphi=\mathbf{1}_{(a,b)}$ is as in Examples \ref{Ex:gradientexamples} (i), then $X$ is $(s,1)$-variable w.r.t. $\varphi$ if $\int_\mathbb{R} (k_{1-s}(y-a)+k_{1-s}(y-b))\mu_X^I(dy)<+\infty$, that is, if $U^{1-s}\mu_X^I$ is finite at $a$ and $b$. This would imply that the upper $s$-density of $\mu_X^I$ at $a$ is zero, $\limsup_{\varepsilon\to 0} \varepsilon^{-s}\mu_X^I((a-\varepsilon,a+\varepsilon))=0$, and similarly at $b$. See \cite[Definition 2.55]{AFP} or \cite[Definition 6.8]{Mattila} for the notion of upper density. Conversely, if for some $s'>s$ the upper $s'$-density 
of $\mu_X^I$ at $a$ is finite, $\limsup_{\varepsilon\to 0} \varepsilon^{-s'}\mu_X^I((a-\varepsilon,a+\varepsilon))<+\infty$, and similarly at $b$, then $U_{1-s}\mu_X^I$ is finite at $a$ and $b$. The path $X$ may hit $a$ and $b$, but it cannot spend any positive time at these points, and it has to approach and leave them quickly enough.
\end{example}

\begin{example}\label{Ex:LTtovar}
Suppose that $w$ is a strictly positive Borel function on $\mathbb{R}^n$ and that $X$ has local times $L^I_X\in L^r(\mathbb{R}^n, w^r(x)dx)$ for some $1< r< +\infty$. Then $X$ is $(s,p)$-variable w.r.t. any $\varphi\in BV(\mathbb{R}^n)$ such that $U^{1-s}\left\|D\varphi\right\|\in L^{pq}(\mathbb{R}^n, w^{-q}(x)dx)$ with $\frac{1}{q}+\frac{1}{r}=1$, as follows from H\"older's inequality. A similar statement also works for $r=+\infty$ or $r=1$. If $\frac{n}{r}<1-s$, then $U^{1-s}\mu_X^I$ is uniformly bounded, as follows from Proposition \ref{P:LTandupperreg} and \cite[Proposition 4.13]{HTV20}. In this case, the path $X$ is $(s,1)$-variable with respect to any $\varphi\in BV(\mathbb{R}^n)$, as already observed in \cite[Proposition 4.12 and Corollary 4.13]{HTV20}. 
\end{example}

\begin{example}
Suppose that $1<q<+\infty$, $\frac{1}{q}+\frac{1}{r}=1$ and $\gamma_1,\gamma_2>0$ with $\gamma_1+\gamma_2=1-s$. If $\varphi\in BV(\mathbb{R}^n)$ is such that $\left\|D\varphi\right\|\in \dot{L}^q_{-\gamma_1}(\mathbb{R}^n)$ and $X$ is such that $\mu_X^I\in \dot{L}_{-\gamma_2}^r(\mathbb{R}^n)$, then by Proposition \ref{P:convolution} and H\"older's inequality we have 
\[\int_{\mathbb{R}^n} U^{1-s}\left\|D\varphi\right\|\:d\mu_X^I=\int_{\mathbb{R}^n}U^{\gamma_1}\left\|D\varphi\right\| U^{\gamma_2}\mu_X^I\:dx\leq \left\|\left\|D\varphi\right\|\right\|_{\dot{L}^q_{-\gamma_1}(\mathbb{R}^n)} \left\|\mu_X^I\right\|_{\dot{L}_{-\gamma_2}^r(\mathbb{R}^n)},\]
which shows that $X$ is $(s,1)$-variable with respect to $\varphi$. By Proposition \ref{P:convolution} the $(s,2)$-variability of $X$ w.r.t. $\varphi$ requires 
\[\int_{\mathbb{R}^n}\int_{\mathbb{R}^n}\int_{\mathbb{R}^n}k_{\gamma_1}(y-z_1)k_{\gamma_1}(y-z_2)\mu_X^I(dy)\:U^{\gamma_1}\left\|D\varphi\right\|(z_1)U^{\gamma_1}\left\|D\varphi\right\|(z_1)dz_1dz_2<+\infty.\]
If $(z_1,z_2)\mapsto \int_{\mathbb{R}^n}k_{\gamma_1}(y-z_1)k_{\gamma_1}(y-z_2)\mu_X^I(dy)$ is in $L^r(\mathbb{R}^{2n})$ and $\left\|D\varphi\right\|\in \dot{L}^q_{-\gamma_1}(\mathbb{R}^n)$, then this is implied by H\"older's inequality. In contrast to the local time case, the simple application of H\"older's inequality in the present example seem too coarse to be useful.
\end{example}

In cases where $X$ is (a typical realization of) a stochastic process with suitable probability densities, these densities can be used to obtain convenient sufficient conditions for variability in the a.s. sense.

\begin{example}\label{Ex:fBMtovar}
Suppose that $X=B^H$ is the $n$-dimensional fractional Brownian motion $B^H$ with Hurst index $0<H<1$. As observed in Examples \ref{Ex:occumeas} (i) its occupation measure is upper $d$-regular $\mathbb{P}$-a.s. if $0<d\leq n$ is such that $H<\frac{1}{d}$. 
Similarly as in Examples \ref{Ex:LTtovar} this implies that $B^H$ is $(s,1)$-variable w.r.t. any $\varphi\in BV(\mathbb{R}^n)$ if $n-1+s<\frac{1}{H}$, see \cite[Example 4.22]{HTV20} for details. In \cite[Example 4.25]{HTV20} we had used that if $n\geq 2$ and $H>\frac{1}{n}$, then the Gaussian density satisfies 
\[\int_0^T t^{-nH}\exp\left(-\frac{|y|^2}{2t^{2H}}\right)dt\leq c\:|y|^{\frac{1}{H}-n}\]
with $c>0$ depending only on $n$ and $H$, so that for any $s\in (0,1)$ such that $\frac{1}{H}<n-1+s$ we obtain 
\begin{align}\label{E:oldboundfBm}
\mathbb{E}\int_0^T U^{1-s}\left\|D\varphi\right\|(X_t)dt
&\leq c\int_{\mathbb{R}^n}\int_{\mathbb{R}^n} |x-y|^{-n+1-s}|y|^{\frac{1}{H}-n}dy\left\|D\varphi\right\|(dx)\notag\\
&\leq c\int_{\mathbb{R}^n}|x|^{-n+1-s+\frac{1}{H}}\left\|D\varphi\right\|(dx),
\end{align}
where $c>0$ is a constant depending only on $n$, $H$ and $s$. If $\varphi\in BV(\mathbb{R}^n)$ makes the last line in (\ref{E:oldboundfBm}) finite, then $\mathbb{P}$-a.s. $X$ is $(s,1)$-variable w.r.t. $\varphi$. This is a (negative) moment condition for the gradient measure, and similarly as in Example \ref{Ex:indicator} it can be rephrased in terms of upper density of $\left\|D\varphi\right\|$ at the origin. It is a local condition and generally easier to ensure than the global integrability or boundedness of a potential of $\left\|D\varphi\right\|$. For instance, if $A\subset\mathbb{R}^n$ is a set of finite perimeter, \cite[Definition 3.35]{AFP}, and $\mathcal{L}^n(A)<+\infty$, then $\varphi=\mathbf{1}_A$ is in $BV(\mathbb{R}^n)$ and $S_\varphi$ is its essential boundary, \cite[Example 3.68]{AFP}. If $S_\varphi$ has positive distance from the origin, then (\ref{E:oldboundfBm}) is finite for this $\varphi$. To discuss $(s,2)$-variability one can use a variant of (\ref{E:oldboundfBm}): Proceeding similarly as there and using Proposition \ref{P:convolution}, we see that
\begin{multline}\label{E:newboundfBM}
\mathbb{E}\int_0^T (U^{1-s}\left\|D\varphi\right\|(X_t))^2dt\leq c\int_{\mathbb{R}^n} (U^{1-s}\left\|D\varphi\right\|(y))^2|y|^{\frac{1}{H}-n}dy\\
= c\int_{\mathbb{R}^n}\int_{\mathbb{R}^n}\left(\int_{\mathbb{R}^n}k_{1-s}(z_1-y)k_{1-s}(z_2-y)|y|^{\frac{1}{H}-n}dy\right)\left\|D\varphi\right\|(dz_1)\left\|D\varphi\right\|(dz_2).
\end{multline}
If $\varphi\in BV(\mathbb{R}^n)$ makes (\ref{E:newboundfBM}) finite, then $\mathbb{P}$-a.s. $X$ is $(s,2)$-variable w.r.t. $\varphi$. The finiteness of the right hand side of (\ref{E:newboundfBM}) is both a moment and a smoothness condition for $\left\|D\varphi\right\|$.
Since $w(x)=|x|^{1/H-n}$ is a weight of class $A_2$, (\ref{E:Wolffweighted}) shows that the right hand side of (\ref{E:newboundfBM}) is finite if and only if $W_{1-s,2,w}^{\left\|D\varphi\right\|}$ is in $L^1(\mathbb{R}^n,\left\|D\varphi\right\|)$, where 
\[W_{1-s,2,w}^{\left\|D\varphi\right\|}(x)=\int_0^1 \frac{r^{2(1-s)}\left\|D\varphi\right\|(B(x,r))}{w(B(x,r))}\:\frac{dr}{r},\quad x\in\mathbb{R}^n.\]
Here we have $w(B(x,r))=\int_{B(0,r)} |z+x|^{1/H-n}dz$, so that on the one hand $w(B(0,r))=c\:r^{1/H}$, and on the other hand $w(B(x,r))$ is comparable to $|x|^{1/H-n}r^n$ on $\{|x|\geq 2\}$. If, for instance, $\varphi$ is constant on $B(0,3)$, then $W_{1-s,2,w}^{\left\|D\varphi\right\|}(x)=0$ for $|x|<2$, and on $\{|x|\geq 2\}$ the quantity  $W_{1-s,2,w}^{\left\|D\varphi\right\|}(x)$ is dominated by a multiple of $|x|^{1/H-n} W^{\left\|D\varphi\right\|}_{1-s,2}(x)$. This could be in $L^1(\mathbb{R}^n,\left\|D\varphi\right\|)$ even if $\varphi$ has singularities (of a similar type as in Example \ref{Ex:gradientexamples} (ii)) on $\{|x|\geq 2\}$.
\end{example}

\begin{example}\label{Ex:Levytovar}
For L\'evy processes (and even for some more general Markov processes) with values in $\mathbb{R}^n$ calculations similar to those in Example \ref{Ex:fBMtovar} can be based on known scaling properties or upper heat kernel estimates. To give a concrete example, suppose that $X$ is the isotropic $\alpha$-stable L\'evy process, where $0<\alpha<2$. In view of Examples \ref{Ex:LT} (ii) and \ref{Ex:LTtovar} we may assume $n\geq 2$. Writing $p(t,y)$ for the density of $X_t$, we have
\[\int_0^T p(t,y)dt\leq c\:|y|^{\alpha-n}\] 
with a constant $c>0$ depending only on $n$ and $\alpha$, and for $s\in (0,1)$ such that $\alpha<n-1+s$ also 
\[\mathbb{E}\int_0^T U^{1-s}\left\|D\varphi\right\|(X_t)dt\leq c\int_{\mathbb{R}^n}|x|^{-n+1-s+\alpha}\left\|D\varphi\right\|(dx)\]
with $c>0$ depending only on $n$, $\alpha$ and $s$. If $\varphi$ makes the right hand side finite, then $X$ is $\mathbb{P}$-a.s. $(s,1)$-variable w.r.t. $\varphi$. In a similar manner as before we see that $X\in V(s,p,\varphi)$ is $\mathbb{P}$-a.s. if $\varphi$ satisfies
\begin{equation}\label{E:newboundLevy}
\int_{\mathbb{R}^n} (U^{1-s}\left\|D\varphi\right\|(y))^p|y|^{\alpha-n}dy<+\infty.
\end{equation}
\end{example}

A first consequence of variability is the following, cf. \cite[Lemma 2.4]{HTV20}. It is immediate from \cite[Corollary 4.4 and its proof]{HTV20}. Recall that we denote the approximate discontinuity set of $\varphi$ by $S_\varphi$.
\begin{proposition}\label{P:welldef}
Let $\varphi\in BV(\mathbb{R}^n)$, $p\in [1,+\infty]$, $s\in (0,1)$ and suppose that $X\in V(\varphi,s,p)$. Then the set $\{t\in [0,T]: X_t\in S_\varphi\}$ is of zero $\mathcal{L}^1$-measure, and for any two Lebesgue representatives $\widetilde{\varphi}_1$ and $\widetilde{\varphi}_2$ of $\varphi$ we have $\widetilde{\varphi}_1(X_t)=\widetilde{\varphi}_2(X_t)$ for $\mathcal{L}^1$-a.e. $t\in (0,T)$.  
\end{proposition}

Proposition \ref{P:welldef} ensures that the following definition is correct.

\begin{definition}\label{D:composition}
Let the hypotheses of Proposition \ref{P:welldef} be in force. We define the \emph{composition $\varphi\circ X$ of $\varphi$ and $X$} to be the $\mathcal{L}^1$-equivalence class on $(0,T)$ of 
\[t\mapsto\widetilde{\varphi}(X_t),\] 
where $\widetilde{\varphi}$ is an arbitrary Lebesgue representative of $\varphi$. 
\end{definition}

\begin{remark}\label{R:comparetoaverage}
The $(s,p)$-variability of a path $X$ w.r.t. $\varphi\in BV(\mathbb{R}^n)$ does generally not imply any variability of $X+x$ w.r.t. $\varphi$ if $x\in \mathbb{R}^n\setminus \{0\}$. If for all $x\in\mathbb{R}^n$ the path $X+x$ is $(s,1)$-variable w.r.t. $\varphi$ and each $\varphi\circ (X+x)$ is in $L^1(0,T)$, then with an adequate interpretation of the integral one can view 
\begin{equation}\label{E:comparetoaverage}
t\mapsto \int_0^t\varphi(X_s+\cdot)\:ds
\end{equation}
as an absolutely continuous $BV(\mathbb{R}^n)$-valued function on $[0,T]$; this involves a standard redefinition procedure. If, on the other hand, (\ref{E:comparetoaverage}) is an absolutely continuous $BV(\mathbb{R}^n)$-valued function, then its derivative in the $\mathcal{L}^1$-a.e. sense is $t\mapsto \varphi\circ (X_t+\cdot)$. Although it does not fit it rigorously, this perspective is very much related to \cite[Lemma 20 iv.]{GG20b}.
\end{remark}

If the path $X$ itself defines an element of some Sobolev space and it is also variable w.r.t. $\varphi$, then we can exploit both these facts together to 
see that also $\varphi\circ X$ is an element of a certain Sobolev space. This is formulated in the next theorem, which is our main result. It generalizes \cite[Theorem 2.13 (i)]{HTV20} in the sense that $X$ does not have to be H\"older continuous. 
In the following we consider Sobolev and H\"older spaces over $U=(0,T)\subset \mathbb{R}$ and use (\ref{E:Gagliardo}) with measure $\nu=\mathcal{L}^1|_{(0,T)}$. Recall that we write $W^{\theta,\infty}(0,T;\mathbb{R}^n):=C^\theta(0,T;\mathbb{R}^n)$. 

\begin{theorem}
\label{thm:sobolev-membership}
Let $\varphi\in BV(\mathbb{R}^n)$, $X$ is a path, $s,\theta \in (0,1)$, $1\leq p<+\infty$ and $1\leq q \leq+\infty$. Suppose that 
\[X \in W^{\theta,q}(0,T;\mathbb{R}^n)\cap V(\varphi,s,p).\] 
\begin{enumerate}
\item[(i)] If $r\geq 1$ is such that $\frac{1}{p}+\frac{s}{q}\leq \frac{1}{r}$ and $\beta<s\theta$, then $\varphi \circ X \in W^{\beta,r}(0,T)$
and in particular, 
\[[\varphi\circ X]_{\beta,r}\leq c\: [X]_{\theta,q}^s\: \Vert U^{1-s}\left\|D\varphi\right\|(X_\cdot) \Vert_{L^p(0,T)}\]
with a constant $c>0$ depending only on $n$, $s$, $p$, $q$, $r$, $\theta$, $\beta$ and $T$.
\item[(ii)] If $\frac{1}{p}+\frac{s}{q}< s\theta$, then for any $\theta'\leq \theta-\frac{1}{q}$ we have $X \in W^{\theta',\infty}(0,T;\mathbb{R}^n)$,
and for any $\beta < s\theta - \frac{1}{p}-\frac{s}{q}$ we have $\varphi \circ X \in W^{\beta,\infty}(0,T)$.
\end{enumerate}
\end{theorem}

\begin{remark}\mbox{}
\begin{enumerate}
\item[(i)] For our purposes Theorem \ref{thm:sobolev-membership} (i) is of interest, more precisely, the case that $q$ is small enough to have $s\theta\leq \frac{1}{p}+\frac{s}{q}$. 
For larger $q$ statement (ii) tells that we implicitly assumed that $X$ is H\"older continuous, and this case had been discussed in \cite[Theorem 2.13 (i)]{HTV20}. In the limit case $q=+\infty$ we recover \cite[first part of Theorem 2.15]{HTV20}. 
\item[(ii)] If $p=1$, then $r\geq 1$ together with $\frac{1}{p}+\frac{s}{q}\leq \frac{1}{r}$ forces $q=+\infty$ and $r=1$.  In order to show the Sobolev regularity of a composition $\varphi\circ X$  involving a discontinuous path $X$ Theorem \ref{thm:sobolev-membership} therefore requires $X$ to be $(s,p)$-variable w.r.t. $\varphi$ with some $p>1$.
\end{enumerate}
\end{remark}

\begin{example}
Suppose that $I=[0,T]$ and $X:I\to\mathbb{R}$ is Lipschitz with $|X'|$ bounded away from zero $\mathcal{L}^1$-a.e. and $X'$ changes sign only finitely many times on $[0,T]$. Let $\varphi=\mathbf{1}_{(a,b)}$, where $0<a<b<T$. By Examples \ref{Ex:Lipschitzcorr}, \ref{Ex:indicator} and \ref{Ex:LTtovar} the function $X$ is $(s,p)$-variable w.r.t $\varphi$ if $sp<1$. Since we can choose $q=+\infty$ and $\theta$ arbitrarily close to $1$, Theorem \ref{thm:sobolev-membership} can be used to conclude that $\varphi\circ X\in W^{\beta,r}(0,T)$ for any $\beta$ and $r$ such that $\beta r<1$, as expected.
\end{example}

\begin{example}\label{Ex:compofBm}
Suppose that $X=B^H$ is an $n$-dimensional fractional Brownian motion with Hurst index $0<H<1$ and let $0<\theta<H$. If $n=1$, then we can find an event of probability one on which any realization of $X$ is $\theta$-H\"older continuous and has bounded local times. If $I$ and $\varphi$ are as in the preceding example, then we can apply Theorem \ref{thm:sobolev-membership} with $q=+\infty$ and $\theta$ arbitrarily close to $H$. It follows that for any $\beta$ and $r$ with $\beta r<H$ the composition $\varphi\circ X$ is in $W^{\beta,r}(0,T)$ $\mathbb{P}$-a.s. Now suppose that $n\geq 2$, $H>\frac{1}{n}$, $s\in (0,1)$ and that $\varphi\in BV(\mathbb{R}^n)$ is a function that makes \eqref{E:oldboundfBm} finite. Then we can find an event of full probability on which $X$ is both $\theta$-H\"older continuous and $(s,1)$-variable w.r.t. $\varphi$, and Theorem \ref{thm:sobolev-membership}, applied with $p=1$, $q=+\infty$ and $s$ as in \eqref{E:oldboundfBm}, shows that for any $\beta<sH$ we have $\varphi\circ X\in W^{\beta,1}(0,T)$ $\mathbb{P}$-a.s. If $\varphi$ makes \eqref{E:newboundfBM} finite, then we can apply the theorem with $p=2$ to find that for any $\beta<sH$ we have $\varphi\circ X\in W^{\beta,2}(0,T)$ $\mathbb{P}$-a.s.
\end{example}

\begin{example}\label{Ex:compoLevy}
Suppose that $I=[0,T]$ and that $X:I\to \mathbb{R}$ is a symmetric $\alpha$-stable L\'evy process with $1<\alpha<2$. If $1\leq q<\alpha$ and $\theta<\frac{1}{\alpha}$, then by \cite[Th\'eor\`eme VI.1]{CKR93} the continuation by zero of $X$ is an element of the Besov space $B_{q,\infty}^\theta(\mathbb{R})$ $\mathbb{P}$-a.s. By \cite[2.3.2, Proposition 2 (7)]{Triebel83} it follows that $\mathbb{P}$-a.s. the continuation by zero of $X$ is an element of the Besov space $B_{q,q}^\theta(\mathbb{R})$, which coincides with 
$W^{\theta,q}(\mathbb{R})$, see \cite[Section 2.6.1]{Triebel92} or \cite[Section V.5]{Stein70} and \cite[Section 2.3.5]{Triebel83}. Since $X$ has $\mathbb{P}$-a.s. locally bounded local times, it follows that for $\varphi=\mathbf{1}_{(a,b)}$ and any $\beta<\frac{1}{\alpha+1}$ we have $\varphi\circ X\in W^{\beta,1}(0,T)$ $\mathbb{P}$-a.s. To see this, note that if $p>\frac{\alpha+1}{\alpha}$, then we can find some $q<\alpha$ such that $\frac{1}{p}+\frac{1}{pq}<1$, and given any $s<\frac{\alpha}{\alpha+1}$ we can find $p$ satisfying the preceding and $sp<1$. The same conclusion could be reached using \cite[Theorem 1.1]{Schilling97}. A straightforward generalization to other L\'evy processes can be provided using \cite[Theorems 3.2 and 3.3]{Herren97}. Now suppose that $0<\alpha<2$, $\alpha\vee \frac12 <q$ and $0\vee (\frac{1}{q}-1)<\theta<\frac{1}{q}$. Then $X\in B_{q,\infty}^\theta (0,T)$ $\mathbb{P}$-a.s. by \cite[Theorem 1.1]{Schilling97}, and similarly as before this implies $X\in W^{\theta,q}(0,T)$  $\mathbb{P}$-a.s. for $\theta$ in the specified range. If in addition $q>1$ and $\varphi\in BV(\mathbb{R})$ is such that (\ref{E:newboundLevy}) holds with some $p>\frac{q+1}{q}$, then for any $\beta<\frac{s}{q}$ we have $\varphi\circ X\in W^{\beta,1}(0,T)$. We point out that the results on path regularity in \cite{Schilling97} and \cite{Schilling00} apply to general Feller processes.
\end{example}


To prove Theorem \ref{thm:sobolev-membership} we begin with the following well-known and elementary lemma. 
\begin{lemma}
\label{lemma:gagliardo-embedding}
Let $1\leq q<+\infty$ and $\theta\in(0,1)$. Then for any $p\leq q$, $0<\beta<\theta$ and $\mathcal{L}^1$-a.e. $t\in [0,T]$ we have
\[\left[\int_0^T \frac{|X_t-X_\tau|^{p}}{|t-\tau|^{1+\beta p}} d\tau \right]^{\frac{1}{p}} \leq c \left[\int_0^T \frac{|X_t-X_\tau|^{q}}{|t-\tau|^{1+\theta q}} d\tau\right]^{\frac{1}{q}}\]
with a constant $c>0$ depending only on $p,q, T, \beta$ and $\theta$. In particular, if $[X]_{\theta,q} < \infty$, then
$$
t\mapsto \int_0^T \frac{|X_t-X_\tau|^{p}}{|t-\tau|^{1+\beta p}} d\tau \in L^{\frac{q}{p}}(0,T).
$$
\end{lemma}
We recall the simple arguments for the reader's convenience.
\begin{proof}
Let $p<q$ and set $\gamma := 1 + (\beta - \theta)p - \frac{p}{q}$.
Since $\beta<\theta$ and $p<q$, we have $\gamma < \frac{q-p}{q}$. Now H\"older inequality implies
\begin{equation*}
\begin{split}
\int_0^T \frac{|X_t-X_\tau|^{p}}{|t-\tau|^{1+\beta p}} d\tau 
& = \int_0^T |t-\tau|^{-\gamma}\frac{|X_t-X_\tau|^{p}}{|t-\tau|^{1+\beta p - \gamma}} d\tau \\
& \leq \left(\int_0^T |t-\tau|^{-\gamma \frac{q}{q-p}}d\tau \right)^{\frac{q-p}{q}} \left(\int_0^T \frac{|X_t-X_\tau|^{p\frac{q}{p}}}{|t-\tau|^{(1+\beta p-\gamma)\frac{q}{p}}} d\tau\right)^{\frac{p}{q}} \\
&\leq C\left(\int_0^T \frac{|X_t-X_\tau|^{q}}{|t-\tau|^{1+\theta q}} d\tau\right)^{\frac{p}{q}}.
\end{split}
\end{equation*}
This verifies the claim for $p<q$, while the case $p=q$ follows directly from the observation
$|t-\tau|^{-1-\beta p} \leq T^{(\theta - \beta)p}|t-\tau|^{-1-\theta p}$, valid for all $t\neq \tau$.  
\end{proof}

We also recall an immediate consequence of (\ref{E:trivialbound}) and (\ref{E:backendMVT}). As before, the approximate discontinuity set of $\varphi$ is denoted by $S_\varphi$.
\begin{proposition}
\label{prop:basic-bound}
Let $\varphi\in BV(\mathbb{R}^n)$ and suppose that 
$X\in V(\varphi,s,1)$ is a path. Then for any $t,\tau\in [0,T]$ such that $X_t,X_\tau\notin S_\varphi$ we have
$$
|\varphi(X_t)-\varphi(X_\tau)| \leq c |X_t-X_\tau|^{s}[U^{1-s}\left\|D\varphi\right\|(X_t) + U^{1-s}\left\|D\varphi\right\|(X_\tau)].
$$
with a constant $c>0$ depending only on $n$, $s$ and $p$. 
\end{proposition}

The following result is the multiplicative key estimate for Theorem \ref{thm:sobolev-membership}. It is an analog of \cite[Proposition 4.28]{HTV20}. Instead of a H\"older seminorm of the path $X$ used there, it involves a Gagliardo seminorm of $X$, which can also be finite in the discontinuous case. 

\begin{proposition}
\label{prop:key-estimate}
Let $\varphi\in BV(\mathbb{R}^n)$, $s,\theta \in (0,1)$, $1\leq p<+\infty$ and $1\leq q \leq+\infty$. Suppose that $X$ is a path with $X \in W^{\theta,q}(0,T;\mathbb{R}^n)\cap V(\varphi,s,p)$.
Then for any $r\geq 1$ such that $\frac{1}{p}+\frac{s}{q}\leq \frac{1}{r}$ and any $\beta < s\theta$ we have
\begin{equation}\label{E:basicest2}
[\varphi\circ X]_{\beta,r}\leq c\:[X]_{\theta,q}^s \: \Vert U^{1-s}\left\|D\varphi\right\|(X_\cdot) \Vert_{L^p(0,T)}.
\end{equation} 
with $c>0$ depending only on $n,T,s,p,q,r,\beta$ and $\theta$.
\end{proposition}
\begin{proof}
Since the case $q=\infty$ is already covered in \cite{HTV20} it suffices to consider $q<\infty$. 
Using Proposition \ref{prop:basic-bound}, the symmetry of the integrand, Lemma \ref{lemma:gagliardo-embedding} and H\"older inequality we obtain
\begin{equation*}
\begin{split}
[\varphi\circ X]_{\beta,r}^r &= \int_0^T\int_0^T \frac{|\varphi(X_t)-\varphi(X_\tau)|^r}{(t-\tau)^{1+\beta r}}\:d\tau dt\\
&\leq c\int_0^T [U^{1-s}\left\|D\varphi\right\|(X_t)]^r \int_0^T \frac{|X_t-X_\tau|^{sr}}{|t-\tau|^{1+\beta r}} d\tau dt \\
&\leq c\int_0^T [U^{1-s}\left\|D\varphi\right\|(X_t)]^r \left(\int_0^T \frac{|X_t-X_\tau|^{q}}{|t-\tau|^{1+\theta q}} d\tau\right)^{\frac{sr}{q}} dt \\
&\leq c\left(\int_0^T [U^{1-s}\left\|D\varphi\right\|(X_t)]^{\frac{qr}{q-sr}} dt\right)^{\frac{q-sr}{q}}\left(\int_0^T\int_0^T \frac{|X_t-X_\tau|^{q}}{|t-\tau|^{1+\theta q}} d\tau dt \right)^{\frac{sr}{q}} \\
&= c\:\Vert U^{1-s}\left\|D\varphi\right\|(X_\cdot) \Vert^r_{L^{\frac{qr}{q-sr}}(0,T)} [X]_{\theta,q}^{sr}.
\end{split}
\end{equation*}
From $\frac{1}{p}+\frac{s}{q}\leq \frac{1}{r}$ it follows that $\frac{qr}{q-sr} \leq p$, and hence
\[\Vert U^{1-s}\left\|D\varphi\right\|(X_\cdot) \Vert_{L^{\frac{qr}{q-sr}}(0,T)} \leq c\Vert U^{1-s}\left\|D\varphi\right\|(X_\cdot) \Vert_{L^p(0,T)}.\]
\end{proof}

To conclude the membership of $\varphi \circ X$ in $W^{\beta,r}(0,T)$ it now remains to verify its $r$-integrability.

\begin{proposition}
\label{prop:lp-membership}
Let $\varphi\in BV(\mathbb{R}^n)$, $s \in (0,1)$, $1\leq p<+\infty$ and $1\leq q \leq+\infty$. Suppose that $X$ is a path such that
\[X \in L^q(0,T;\mathbb{R}^n)\cap V(\varphi,s,p).\] 
Then for any $r\geq 1$ such that $\frac{1}{p}+\frac{s}{q}\leq \frac{1}{r}$ we have
$\varphi \circ X \in L^{r}(0,T)$.
\end{proposition}

The argument is as in \cite[Lemma 4.30]{HTV20}.

\begin{proof}
Again it suffices to consider the case $q<\infty$ as $q=\infty$ is already covered by \cite[Lemma 4.30]{HTV20}.
By \cite[Corollary 4.4]{HTV20} we can choose $t_0 \in [0,T]$ such that $X_{t_0} \in \mathbb{R}^n\setminus S_\varphi$ and 
\[M(t_0):=\max\left\lbrace |\varphi(X_{t_0})|, U^{1-s}\left\|D\varphi\right\|(X_{t_0})\right\rbrace < \infty.\]
Clearly $|\varphi(X_t)|^r \leq 2^{r-1}\left[|\varphi(X_t)-\varphi(X_{t_0})|^r+|\varphi(X_{t_0})|^r\right]$
for $\mathcal{L}^1$-a.e. $t\in [0,T]$, and for $t$ such that $X_t\notin S_{\varphi}$ Proposition \ref{prop:basic-bound} implies that
\begin{align}
|\varphi(X_t)-\varphi(X_{t_0})|^r & \leq c|X_t-X_{t_0}|^{sr}[U^{1-s}\left\|D\varphi\right\|(X_t)^r + U^{1-s}\left\|D\varphi\right\|(X_{t_0})^r]\notag\\
&\leq c(|X_t|^{sr}+M(t_0)^{sr})[U^{1-s}\left\|D\varphi\right\|(X_t)^r+M(t_0)^r].\notag
\end{align}
Since $r\leq p$ and $rs\leq q$ we therefore obtain $\int_0^T |\varphi(X_t)|^r dt<+\infty$, provided that 
\[\int_0^T |X_t|^{sr}U^{1-s}\left\|D\varphi\right\|(X_t)^r dt < \infty.\]
But the finiteness of this integral can be seen using H\"older's inequality,
\begin{align}
\int_0^T |X_t|^{sr}U^{1-s}\left\|D\varphi\right\|(X_t)^r dt &\leq \left(\int_0^T U^{1-s}\left\|D\varphi\right\|(X_t)^{\frac{qr}{q-sr}}dt\right)^{\frac{q-sr}{q}} \left(\int_0^T |X_t|^{q}dt\right)^{\frac{sr}{q}} \notag\\ 
&\leq c\Vert U^{1-s}\left\|D\varphi\right\|(X_\cdot) \Vert^r_{p} \Vert X \Vert^{sr}_{q}.\notag
\end{align}
Note that the last inequality holds since $\frac{qr}{q-sr} \leq p$. The result now follows.
\end{proof}

We prove Theorem \ref{thm:sobolev-membership}.
\begin{proof}[Proof of Theorem \ref{thm:sobolev-membership}]
Item (i) follows from Proposition \ref{prop:key-estimate} and Proposition \ref{prop:lp-membership}. Item (ii) then follows directly from the Sobolev embedding theorem, \cite[Theorem 8.2]{DiNezza}: Choosing $r$ such that $\frac{1}{r} = \frac{1}{p}+\frac{s}{q}$ and $\beta<s\theta$ such that $r\beta>1$ we obtain $\varphi \circ X \in W^{\beta - \frac{1}{p}-\frac{s}{q},\infty}(0,T)$. The H\"older continuity of $X$ follows similarly, note that $\theta q > 1 + \frac{q}{ps} > 1$ by the hypotheses in (ii).
\end{proof}

\section{Existence of generalized Stieltjes integrals}\label{S:integral}

As an application we illustrate how Theorem \ref{thm:sobolev-membership} can be used to guarantee the existence of generalized Stieltjes integrals in the sense of \cite{Zahle98, Zahle01}.

Let $T>0$, let $f,g:[0,T]\to \mathbb{R}$ be Borel functions such that $f$ is right-continuous at $0$ and $g$ left-continuous at $T$. Set 
\[f_0(t) = f(t) - f(0)\quad \text{ and }\quad g_T(t) = g(t) - g(T), \quad t\in [0,T].\] 
Let $\alpha \in (0,1)$ and suppose that the left-sided Weyl--Marchaud derivative 
\begin{equation}
\label{eq:left-derivative}
 D_{0+}^\alpha f_0(t)
 = \frac{1}{\Gamma(1-\alpha)}\left( \frac{f(t)-f(0)}{t^\alpha} + \alpha \int_0^t \frac{f(t)-f(u)}{(t-u)^{\alpha+1}}  du \right), \quad t\in (0,T]
\end{equation}
of $f_0$ of order $\alpha$ exists as a member of $L^p(0,T)$ and the right-sided Weyl--Marchaud derivative 

\[
 D_{T-}^{1-\alpha} g_T(t)
 = \frac{(-1)^{1-\alpha}}{\Gamma(\alpha)}\left( \frac{g(t)-g(T)}{(T-t)^{1-\alpha}} + (1-\alpha) \int_t^T \frac{g(t)-g(u)}{(u-t)^{2-\alpha}}  du \right) \quad t\in[0,T)
\]
of $g_T$ of order $1-\alpha$ exists as a member of $L^q(0,T)$, where $\frac{1}{p}+\frac{1}{q}\leq 1$. Then the generalized Stieltjes integral in the sense of \cite{Zahle98}, defined as 
\begin{equation}\label{eq:ZSIntegral}
\int_0^T f(u) dg(u):= (-1)^\alpha \int_0^T D^{\alpha}_{0+} f_0(t) \, D^{1-\alpha}_{T-} g_T(t)dt + f(0)(g(T)-g(0)),
\end{equation}
exists, and its value does not depend on the choice of $\alpha$, see \cite[Section 2]{Zahle98}. 
\begin{remark}\label{R:simplify}
If $\alpha p<1$ then the right-continuity of $f$ at $0$ can be dropped and the right hand side in (\ref{eq:ZSIntegral}) can be replaced by $(-1)^\alpha \int_0^T D^{\alpha}_{0+} f(t) \, D^{1-\alpha}_{T-} g_T(t)dt$. See \cite{Zahle01}.
\end{remark}
The following is a variant of the well-known existence result for (\ref{eq:ZSIntegral}), phrased in terms of the Sobolev spaces $W^{\gamma,p}(0,T)$. The membership of the Borel function $f_0$ in a Sobolev space $W^{\gamma,p}(0,T)$ refers to the class it defines, and similarly for $g_T$.

\begin{theorem}\label{thm:sobolev-class-integral}
Let $\gamma,\delta\in(0,1)$ be such that $\gamma+\delta>1$ and $1\leq p,q \leq +\infty$ such that 
\begin{equation}\label{E:condleqone}
\frac{1}{p}+\frac{1}{q} < \gamma +\delta.
\end{equation}
Suppose that $f,g:[0,T]\to \mathbb{R}$ are Borel functions right- resp. left-continuous at $0$ and $T$ and satisfying $f\in W^{\gamma,p}(0,T)$ and $g\in W^{\delta,q}(0,T)$. Then the integral $\int_0^T f(u) dg(u)$, defined as in \eqref{eq:ZSIntegral}, exists. Moreover, 
there is a constant $c>0$ depending on $p$, $q$, $\gamma$, $\delta$ and $T$ such that
\begin{equation}
\label{eq:integral-bound}
\left|\int_0^T f(u)dg(u) - f(0)\left(g(T)-g(0)\right)\right| \leq c\Vert f_0\Vert_{W^{\gamma,p}(0,T)}\Vert g_T\Vert_{W^{\delta,q}(0,T)}.
\end{equation}
\end{theorem}

\begin{remark}\mbox{}\label{remark:continuity}  
\begin{enumerate}
\item[(i)] By (\ref{E:condleqone}) $\gamma p\leq 1$ and $\delta q\leq 1$ can never occur simultaneously. This forces that at least one of the functions $f$ or $g$ has to be continuous, actually H\"older continuous, but only of a very small order. 
A similar observation was made in \cite[Theorem 4.1 and Remark 4.1]{FrizSeeger}, where conditions force either $f$ or $g$ to be H\"older continuous of small order. If both $\gamma p>1$ and $\delta q>1$, then both $f$ and $g$ are continuous and by \cite[Theorem 2]{FrizVictoir} have finite $\frac{1}{\gamma}$- respectively $\frac{1}{\delta}$-variation. In this case the integral above exists as a Young integral and (\ref{eq:integral-bound}) implies the classical Young-Loeve estimate.
\item[(ii)] If $\gamma p<1$ then the right-continuity of $f$ at $0$ is not needed, the boundary terms on the left hand side of (\ref{eq:integral-bound}) can be dropped, and $f_0$ on the right hand side can be replaced by $f$. 
\end{enumerate}
\end{remark}

Combined with Theorem \ref{thm:sobolev-membership} we obtain the following result on the existence of integrals for the compositions $\varphi \circ X$, where $\varphi\in BV(\mathbb{R}^n)$, $X$ is a path, and the composition is understood as in Definition \ref{D:composition}.

\begin{corollary}
\label{cor:composition-integral}
Let $\varphi\in BV(\mathbb{R}^n)$, $s,\theta \in (0,1)$, $1\leq p<+\infty$ and $1\leq q \leq+\infty$. Suppose that 
$X \in W^{\theta,q}(0,T;\mathbb{R}^n)\cap V(\varphi,s,p)$, that $X$ is right-continuous at $0$ and $\varphi$ is continuous at $X_0$. If $g \in W^{\delta,v}(0,T)$ for some $v\geq 1$ and $\delta\in(0,1)$ such that 
\begin{equation}\label{E:oldparam}
s\theta+\delta >1 \quad \text{and}\quad \frac{1}{v}+\frac{1}{p}+\frac{s}{q}<s\theta+\delta, 
\end{equation}
then the integral $\int_0^T \varphi(X_u)dg(u)$ in the sense of (\ref{eq:ZSIntegral})
exists. Moreover, it obeys the estimate 
\begin{multline}\label{E:compoest}
\left|\int_0^T \varphi(X_u)dg(u) - \varphi(X_0)\left(g(T)-g(0)\right)\right| \\
\leq c\:\left(\left\|U^{1-s}\left\|D\varphi\right\|(X_\cdot)\right\|_{L^p(0,T)}[X]_{\theta,q}^s+\left\|\varphi(X_\cdot)-\varphi(X_0)\right\|_{L^p(0,T)} \right) \Vert g_T\Vert_{W^{\delta,q}(0,T)}.
\end{multline}
\end{corollary}

\begin{remark}\mbox{}
\begin{enumerate}
\item[(i)] Suppose that $0=t_0<t_1<...<t_N=T$ is a partition of $[0,T]$ and that $g$ and $\varphi(X_\cdot)$ are continuous at  all $t_k$. Then the above estimate, applied to the intervals $[t_{k-1},t_k]$ in place of $[0,T]$, yields 
\begin{multline}
\sum_{k=1}^N \left|\int_{t_{k-1}}^{t_k} \varphi(X_u)dg(u) - \varphi(X_{t_{k-1}})\left(g(t_k)-g(t_{k-1})\right)\right| \\
\leq c\sum_{k=1}^N\:\left(\left\|U^{1-s}\left\|D\varphi\right\|(X_\cdot)\right\|_{L^p(t_{k-1},t_k)}[X]_{\theta,q}^s+\left\|\varphi(X_\cdot)-\varphi(X_{t_{k-1}})\right\|_{L^p(t_{k-1},t_k)} \right) \times\notag\\
\times \Vert g_{t_k}\Vert_{W^{\delta,q}(t_{k-1},t_k)}.
\end{multline}
It follows that $\int_0^T \varphi(X_u)dg(u)$ exists as a limit of forward Riemann-Stieltjes sums along refining partitions for which $g$ and $\varphi(X_\cdot)$ are continuous at all partition points $t_k$.
\item[(ii)] Although (i) ensures the existence of the integral as a limit of Riemann-Stieltjes sums in along \emph{suitable} partitions, this does not work for sequences of arbitrary partitions: Already for the one-dimensional case $n=1$ and 
the Heaviside function $\varphi(x) = \textbf{1}_{\{x>a\}}$ it is known that $\varphi(X)$ is of infinite $p$-variation for all $p\geq 1$ if (and only if) $X$ crosses the level $a$ infinitely often, \cite[Proposition 5.0.3]{Chen}. In this case the integral does not exists as a Young integral, and one can find sequences of partitions along which the corresponding Riemann sums do not converge. This is in line with \cite[Remark 4.1]{FrizSeeger}, where it was observed that integrals of Heaviside functions do not exists as limits of arbitrary Riemann sums. Because these integrands are discontinuous, this does not contradict well-known embeddings for the case of continuous Sobolev functions, \cite[Theorem 2]{FrizVictoir}. Also the compositions we consider above will in general not be continuous and not be of bounded $p$-variation for any $p$.
\end{enumerate}
\end{remark}

\begin{remark}\label{R:simplyalt}
If $s(\theta-\frac{1}{q})\leq \frac{1}{p}$, then the continuity assumptions on $X$ and $\varphi$ in Corollary \ref{cor:composition-integral} can be dropped and (\ref{E:compoest}) can be simplified.
\end{remark}

In the following examples we silently assume that $X$ and $\varphi$ satisfy the continuity assumptions of Corollary \ref{cor:composition-integral} or, alternatively, that the condition in Remark \ref{R:simplyalt} holds.

\begin{example}
\label{example:previous-results}
Let $g\in W^{\delta,\infty}(0,T)$ with some $\delta>1-s\theta$. Then the integral $\int_0^T \varphi(X_u)dg(u) $ exists, provided that $X \in W^{\theta,q}(0,T;\mathbb{R}^n)\cap V(\varphi,s,p)$, where $\frac{1}{p}+\frac{s}{q}<s\theta+\delta$. Corollary \ref{cor:composition-integral} generalizes \cite[Theorem 2.13 (ii)]{HTV20} in which the case $p=1$, $q=\infty$ was covered. 
\end{example}

\begin{example}
Let $X \in W^{\theta,\infty}(0,T;\mathbb{R}^n)\cap V(\varphi,s,2)$. Then $\varphi \circ X \in W^{\beta,2}(0,T)$ for any $\beta < s\theta$ and consequently the integral $\int_0^T \varphi(X_u)dg(u)$ exists, provided that $g \in W^{\delta,v}(0,T)$ with $\delta>1-s\theta$ and $\frac{1}{v}<s\theta+\delta$. If $\delta$ is as in the first condition, then the second always holds for $v=2$.
\end{example}

\begin{example}
Suppose $X \in W^{\theta,q}(0,T;\mathbb{R}^n)\cap V(\varphi,s,p)$. By Theorem \ref{thm:sobolev-membership} we have
$\varphi \circ X \in W^{\beta,r}(0,T)$ for any $\beta < s\theta$ and $\frac{1}{p}+\frac{s}{q}\leq \frac{1}{r}$. In particular, if $q=\infty$, $p=2$, and $s\theta > \frac12$, then $\varphi \circ X \in W^{\beta,2}(0,T)$ for some $\beta>\frac12$, and hence Theorem \ref{thm:sobolev-class-integral} guarantees the existence of  the integral $\int_0^T \varphi(X_u)d \varphi(X_u)$.
Note that now, by Theorem \ref{thm:sobolev-membership}, we have $\varphi \circ X \in W^{\beta,\infty}(0,T)$ for any $\beta < s\theta - \frac12$. This means that $\varphi \circ X$ is forced to be H\"older continuous, but only of a very small order.
\end{example}

\begin{remark}
Let $F$ be an element of the Sobolev space $W^{1,1}(\mathbb{R}^n)$ and such that $\partial_k F \in BV(\mathbb{R}^n)$ for $k=1,\ldots,n$. One can then follow the proof of \cite[Theorem 2.14.]{HTV20} to obtain
$$
F(X_T) = F(X_0) + \sum_{k=1}^n \int_0^T \partial_k F(X_u)dX_u^k,
$$
provided that for $k=1,\ldots,n$ we have $X\in W^{\theta,q}(0,T;\mathbb{R}^n) \cap V(\partial_k F,s,p)$ for some $\theta > \frac{1}{1+s}$ and $\frac{1}{p} + \frac{1+s}{q}\leq 1$. This generalizes the change of variable formula \cite[Theorem 2.14.]{HTV20} (which covers the case $p=1$, $q=\infty$) and can be used to construct solutions to differential systems involving $BV$-coefficients as in \cite[Section 3]{HTV20}.
\end{remark}

To prove Theorem \ref{thm:sobolev-class-integral} we use the following embedding of Hardy type.
\begin{proposition}
\label{prop:boundary-terms}
Suppose $f:[0,T]\to\mathbb{R}$ is a Borel function satisfying $f\in W^{\beta,p}(0,T)$ for some $\beta \in (0,1)$ and $p\geq 1$ such that $\beta p \neq 1$. Then there exists a constant $c>0$ depending solely on $\beta$, $p$, and $T$, such that 
\begin{equation}
\label{eq:key-boundary}
\int_{0}^T \frac{|f(t)-f(0)|^p}{t^{\beta p}}dt \leq c\left[\int_0^T \int_0^T \frac{|f(t)-f(u)|^p}{|t-u|^{1+\beta p}}\:dudt + \int_0^T |f(t)-f(0)|^p dt\right].
\end{equation}
\end{proposition}

\begin{proof}
In the case $\beta p<1$ inequality (\ref{eq:key-boundary}) follows from \cite[Lemma 4.33]{HTV20} and smooth approximation by mollification. In the case $\beta p>1$ inequality (\ref{eq:key-boundary}) follows using \cite[Theorem 5.9]{KufnerPersson}.
\end{proof}

We prove Theorem \ref{thm:sobolev-class-integral}.
\begin{proof}
Assume first that 
\begin{equation}\label{E:condsimple}
\frac{1}{p}+\frac{1}{q} \leq 1.
\end{equation}
Let $\alpha \in (1-\theta,\gamma)$. By \eqref{eq:ZSIntegral} we have 
\[\left|\int_0^T f(u)dg(u) - f(0)(g(T)-g(0))\right| = \left|\int_0^T D^{\alpha}_{0+} f_0(t) \, D^{1-\alpha}_{T-} g_T(t)dt\right|,
\]
and by H\"older's inequality it suffices to show that
\begin{equation}
\label{eq:derivative-norm}
\int_0^T \left\vert D_{0+}^{\alpha}f_0(t)\right\vert^p dt \leq C\Vert f_0\Vert_{W^{\gamma,p}(0,T)}^p\quad \text{and}\quad \int_0^T \left\vert D_{T-}^{1-\alpha}g_T(t)\right\vert^q dt \leq C\Vert g_T\Vert^q_{W^{\theta,q}(0,T)}.
\end{equation}
By Lemma \ref{lemma:gagliardo-embedding} we have 
\begin{equation*}
\begin{split}
\int_0^T \left|D_{0+}^\alpha f_{0}\right|^p dt &\leq 2^{p-1}\left[ \int_0^T \frac{|f(t)-f(0)|^p}{t^{\alpha p}} dt + \int_0^T \left(\int_{0}^T 
\frac{|f(t)-f(\tau)|}{|t-\tau|^{1+\alpha}}d\tau\right)^p dt \right]\\
&\leq c\left[ \int_0^T \frac{|f(t)-f(0)|^p}{t^{\alpha p}} dt + \int_0^T \int_{0}^T 
\frac{|f(t)-f(\tau)|^p}{|t-\tau|^{1+\gamma p}}d\tau dt \right].
\end{split}
\end{equation*}
If $\gamma p \neq 1$, then by Proposition \ref{prop:boundary-terms} we have 
\[\int_0^T \frac{|f(t)-f(0)|^p}{t^{\alpha p}} dt \leq T^{(\gamma-\alpha)p}\int_0^T \frac{|f(t)-f(0)|^p}{t^{\gamma p}} dt\leq c\:\Vert f_0\Vert_{W^{\gamma,p}(0,T)}^p.\]
If $\gamma p =1$, then $\alpha p < 1$ and we have $\Vert f_0 \Vert_{W^{\alpha,p}(0,T)} \leq c\Vert f_0\Vert_{W^{\gamma,p}(0,T)}$,  \cite[Proposition 2.1]{DiNezza}, and using Proposition \ref{prop:boundary-terms} with $\alpha$ in place of $\beta$, we obtain
\[\int_0^T \frac{|f(t)-f(0)|^p}{t^{\alpha p}} dt \leq c\Vert f_0\Vert_{W^{\alpha,p}(0,T)} \leq c\Vert f_0\Vert_{W^{\gamma,p}(0,T)}.\]
This shows the first inequality in (\ref{eq:derivative-norm}), the second follows by similar arguments.

Now suppose that 
\begin{equation}\label{E:newcond}
1< \frac{1}{p}+\frac{1}{q} <\gamma+\delta.
\end{equation}
By the elementary continuous embedding of $W^{\gamma,\infty}(0,T)$ into $W^{\gamma',r}(0,T)$, valid for any $0<\gamma'<\gamma<1$ and $1\leq r<+\infty$, we can make $\gamma$ and $\delta$ slightly smaller if necessary and assume that both $p$ and $q$ are finite. Using the fact that any element of $W^{\gamma,p}(0,T)$ can be continued to an element of $W^{\gamma,p}(\mathbb{R})$, together with well-known norm equivalences, \cite[Section 2.6.1]{Triebel92}, and embeddings, \cite[2.7.1 Theorem]{Triebel83}, we may also assume that $p>1$ and $q>1$, again to the cost of making $\gamma$ or $\delta$ a bit smaller if necessary. Now let $q'$ be such that $\frac{1}{q}+\frac{1}{q'}=1$. Then by the left inequality in (\ref{E:newcond}) we have $p<q'$, and by the results just cited the embedding of $W^{\gamma,p}(0,T)$ into $W^{\gamma',q'}(0,T)$ with $\gamma'=\gamma-\frac{1}{p}+\frac{1}{q'}$ is continuous. So if $f_0\in W^{\gamma,p}(0,T)$, the result follows from the first part of the proof with $\gamma'$ and $q'$ in place of $\gamma$ and $p$, note that by the right inequality in (\ref{E:newcond}) we have $\gamma'+\delta>1$.

\end{proof}

\section{Berman's inequality revisited}\label{S:Berman}

In Proposition \ref{prop:key-estimate} and in \cite[Proposition 4.28]{HTV20} we used both the regularity of $X$ and the regularity of $\mu_X^I$. In this auxiliary section, which is independent of the preceding sections, we briefly revisit a result due Berman \cite{Berman69} that provides a quantitative limitation for the possible simultaneous regularity of a path $X$ and its occupation measure. In particular for low space dimensions this limitation is not too severe. Natural and far reaching dimension-free results are provided in \cite[Theorem 31 and Section 5.2]{GG20b}.

For a function $X:I\to\mathbb{R}$ with square integrable local times on a bounded interval $I\subset \mathbb{R}$ Berman proved in \cite[Lemma 3.1]{Berman69} a lower estimate for the diameter of the range $X(I)$ in terms of the $\dot{L}^2_\alpha(\mathbb{R}^n)$-norm of its local time with some $\alpha\geq 0$. He then used this estimate to show that if, very roughly speaking, this norm decays fast enough as the interval $I$ is made smaller and smaller and the function $X$ itself is H\"older continuous on $I$ of order $\gamma\in (0,1)$, then we must have 
\begin{equation}\label{E:restrict}
\gamma<\frac{1}{2\alpha+1}.
\end{equation}
For $\alpha>0$, and in particular for the case $\alpha\geq 1$ of weakly differentiable local times, (\ref{E:restrict}) restricts the possible range of $\gamma$. However, for $\alpha=0$ condition (\ref{E:restrict}) is no restriction on $\gamma\in (0,1)$. The same is true for the case that the occupation measures have no densities, but themselves are elements of a Sobolev space $\dot{L}^2_\alpha(\mathbb{R}^n)$ of suitable negative order $\alpha< 0$. This case had not been considered in \cite[Lemma 3.1]{Berman69}, although the mechanism allows it and the result is very natural. Before stating Berman's inequality for occupation measures in Corollary \ref{C:Berman} below, we provide a corresponding inequality for general finite Borel measures.
\begin{theorem}\label{T:Berman}
Let 
\begin{equation}\label{E:alpharange}
1<p\leq  +\infty\quad \text{and}\quad -\frac{n}{p} < \alpha<n-\frac{n}{p}.
\end{equation}
Then there is a constant $c>0$, depending only on $n$, $p$ and $\alpha$, such that for any finite Borel measure $\mu$ on $\mathbb{R}^n$, any $x\in \mathbb{R}^n$ and any $r>0$ we have 
\begin{equation}\label{E:BermanPoincare}
\mu(B(x,r))\leq c\:r^{\alpha+\frac{n}{p}}\left\||\xi|^{\alpha}\hat{\mu}\right\|_{L^p(\mathbb{R}^n)}.
\end{equation}
\end{theorem}


\begin{remark}\label{R:trivial} For nonzero $\mu$ we have $\left\||\xi|^{\alpha}\hat{\mu}\right\|_{L^p(\mathbb{R}^n)}=+\infty$ if $1<p<+\infty$ and $\alpha\leq -\frac{n}{p}$ respectively $p=+\infty$ and $\alpha<0$. Consequently (\ref{E:BermanPoincare}) trivially holds for such $p$ and $\alpha$. For $p=+\infty$ and $\alpha=0$ it is trivial with $c=(2\pi)^{n/2}$. 
\end{remark}

Theorem \ref{T:Berman} follows by similar arguments as used in \cite[p. 273-274]{Berman69}.
\begin{proof}
Since the right hand side of (\ref{E:BermanPoincare}) does not depend on $x$, we may assume that $x=0$. Let $\psi\in C_c^\infty(\mathbb{R}^n)$ be nonnegative, with $\supp\psi\subset B(0,2)$ and such that $\psi(x)= 1$ if $|x|\leq 1$ and set $\psi_r(x):=\psi(\frac{x}{r})$. Since $\mu\in \mathcal{S}'(\mathbb{R}^n)$ and $\psi_r\in \mathcal{S}(\mathbb{R}^n)$, we have 
\[\mu(B(0,r))\leq\int_{\mathbb{R}^n}\psi_r\:d\mu=\int_{\mathbb{R}^n}\hat{\psi}_r(\xi)\hat{\mu}(\xi)\:d\xi
\leq \left(\int_{\mathbb{R}^n}|\xi|^{-\alpha p'}\big|\hat{\psi}_r(\xi)\big|^{p'} d\xi\right)^{\frac{1}{p'}}\left\||\xi|^{\alpha}\hat{\mu}\right\|_{L^p(\mathbb{R}^n)},\]
where $\frac{1}{p}+\frac{1}{p'}=1$; the obvious modifications apply for the case $p=+\infty$. Substituting $\eta=r\xi$ and using $\hat{\psi}_r(\xi)=r^n\hat{\psi}(r\xi)$, the first factor on the right hand side is seen to be $r^{\alpha+\frac{n}{p}}\left(\int_{\mathbb{R}^n}|\eta|^{-\alpha p'}|\hat{\psi}|^{p'}\:d\eta\right)^{1/p'}$. Since this integral converges, the result follows.
\end{proof}

\begin{remark}\label{R:commentsBerman}\mbox{}
\begin{enumerate}
\item[(i)] For $p=2$ and $\alpha=1$ inequality (\ref{E:BermanPoincare}) basically follows from Cauchy-Schwarz, Poincar\'e's inequality and Plancherel's theorem, even if $n=1,2$: Suppose that $\mu$ is absolutely continuous with density $f$ having compact support inside $B(x,r)$. If $\left\||\xi|\hat{\mu}\right\|_{L^2(\mathbb{R}^n)}$ is finite, then it equals $\left\||\nabla f|\right\|_{L^2(\mathbb{R}^n)}$ and
\[\mu(B(x,r))\leq \mathcal{L}^n(B(x,r))^{1/2}\left\|f\right\|_{L^2(B(x,r))}\leq c\:r^{1+\frac{n}{2}}\:\left\||\xi|\hat{\mu}\right\|_{L^2(\mathbb{R}^n)};\]
if it is not finite, the inequality is true anyway. 
\item[(ii)] Suppose that $1<p<+\infty$ and $-\frac{n}{p}<\alpha<0$. Then for $\psi$ as in the proof we have $|\xi|^{-\alpha}\hat{\psi}\in \mathcal{S}(\mathbb{R}^n)$, and we obtain the well-known duality estimate
\begin{multline}\label{E:HY}
\mu(B(0,r))\leq\int_{\mathbb{R}^n}\psi_r\:d\mu=\int_{\mathbb{R}^n}\left(|\xi|^{-\alpha}\hat{\psi}\right)^\vee(x) \left(|\xi|^{\alpha}\hat{\mu}\right)^\vee(x) dx\\
\leq \left\|\psi_r\right\|_{\dot{L}_{-\alpha}^p(\mathbb{R}^n)}\left\|\mu\right\|_{\dot{L}_{\alpha}^{p'}(\mathbb{R}^n)}=c_\psi\:r^{\alpha+\frac{n}{p}}\left\|\mu\right\|_{\dot{L}_{\alpha}^{q'}(\mathbb{R}^n)},
\end{multline}
where $c_\psi=\left\|\psi\right\|_{\dot{L}_{-\alpha}^p(\mathbb{R}^n)}$ and $\frac{1}{p}+\frac{1}{p'}=1$. By translation invariance $B(0,r)$ can again be replaced by $B(x,r)$. For $1< p\leq 2$ we have
$\left\|\mu\right\|_{\dot{L}_\alpha^{p'}(\mathbb{R}^n)}\leq \left\||\xi|^{\alpha}\hat{\mu}\right\|_{L^p(\mathbb{R}^n)}$ by 
Hausdorff-Young, so that (\ref{E:BermanPoincare}) is implied by (\ref{E:HY}); this gives an alternative proof of (\ref{E:BermanPoincare}) for $p$ and $\alpha$ as stated. If $2\leq p<+\infty$, then $\left\||\xi|^{\alpha}\hat{\mu}\right\|_{L^p(\mathbb{R}^n)}\leq \left\|\mu\right\|_{\dot{L}_\alpha^{p'}(\mathbb{R}^n)}$, in this case (\ref{E:BermanPoincare}) implies (\ref{E:HY}). Moreover, (\ref{E:HY}) remains valid if we allow $\psi(x)\geq 1$ for $|x|\leq 1$, and taking the infimum over all such $\psi$, we arrive at the known inequality
\[\mu(B(x,r))\leq \dot{C}_{-\alpha,p}(B(0,1))^{\frac{1}{p}}\:r^{\alpha+\frac{n}{p}}\left\|\mu\right\|_{\dot{L}_{\alpha}^{p'}(\mathbb{R}^n)},\]
where $\dot{C}_{-\alpha,p}$ denotes the \emph{$(-\alpha,p)$-Riesz capacity}, see \cite[Definition 2.2.6]{AH96}. We point out that for $p$, $\alpha$ and $q$ as stated the identity
\begin{equation}\label{E:dualdefcap}
\dot{C}_{-\alpha,p}(A)=\sup\left\lbrace \frac{\nu(A)^p}{(I_{p'}^{-\alpha}(\nu))^{p/p'}}:\ \text{$\nu$ nonnegative Radon measure on $A$}\right\rbrace
\end{equation}
holds for any $A\subset \mathbb{R}^n$ Borel,  \cite[Theorems 2.2.7 and 2.5.1]{AH96}.
\end{enumerate}
\end{remark}

Now let $I$ be a bounded interval and $X:I\to \mathbb{R}^n$ a bounded path. The following Corollary \ref{C:Berman} is a low-$\alpha$-version of Berman's inequality from \cite[Lemma 3.1]{Berman69} for $\mathbb{R}^n$-valued functions $X$. It follows from Theorem \ref{T:Berman} with $\mu=\mu_X^I$ and $r=\diam(X(I))$.

\begin{corollary}\label{C:Berman}
Let $p$ and $\alpha$ be as in (\ref{E:alpharange}). Then there is a constant $K>0$, depending only on $n$, $p$ and $\alpha$, such that 
\begin{equation}\label{E:Bermangeneral}
\diam(X(I))^{\alpha+\frac{n}{p}}\geq \frac{K\:\mathcal{L}^1(I)}{\ \left\||\xi|^{\alpha}\hat{\mu}_X^I\right\|_{L^p(\mathbb{R}^n)}}.
\end{equation}
\end{corollary}

\begin{remark}\mbox{}\label{R:Bermangeneral}
\begin{enumerate}
\item[(i)] The original version \cite[Lemma 3.1]{Berman69} considered the case $n=1$, $p=2$, $\alpha\geq 0$, but was able to allow arbitrarily large $\alpha$ by an integration by parts argument. We are mainly interested in small $\alpha$ and therefore accept the upper bound on $\alpha$.
\item[(ii)] For $p=2$ and $-\frac{n}{2}<\alpha<0$ the square of the right hand side in (\ref{E:Bermangeneral}) is the inverse $(-\alpha,2)$-energy (up to the constant $K^2$) of the normalized occupation measure; this is similar as in Remark \ref{R:commentsBerman} (ii), see also Remark \ref{R:Hausdorffdim} below.
\item[(iii)] Corollary \ref{C:Berman} is clearly linked with the results in \cite{GG20b}: Given $\gamma,\varrho>0$, a path $X:I\to \mathbb{R}^n$ is called \emph{$(\gamma,\varrho)$-irregular} if there is a constant $c>0$ such that 
$\left\||\xi|^\varrho \hat{\mu}_X^J\right\|_{L^\infty(\mathbb{R}^n)}\leq c\:\mathcal{L}^1(J)^\gamma$ for any subinterval $J$ of $I$, see \cite[Definition 1.3]{CatellierGubinelli} or \cite[Definition 2]{GG20b}. If $0<\gamma\leq 1$, $0<\varrho<n$ and $X$ is $(\gamma,\varrho)$-irregular, then Corollary \ref{C:Berman} with $p=+\infty$ and $\alpha=\varrho$ yields $\diam(X(J))\geq (K/c)^{1/\varrho}\:\mathcal{L}^1(J)^{(1-\gamma)/\varrho}$
for each subinterval $J$; this is in line with \cite{GG20b}, see for instance Theorem 31 there.
\end{enumerate}
\end{remark}

For any $p>0$ let $V_p(X,I):=\lim_{\delta\to 0}\sup_{\Delta(\mathcal{P})\leq\delta}\sum_{i=1}^N |X(t_i)-X(t_{i-1})|^p$,
the supremum taken over all partitions $\mathcal{P}$ of $I$ of form $a=t_0<t_1<...<t_N=b$ and with mesh $\Delta(\mathcal{P}):=\max_{1\leq i\leq N}|t_i-t_{i-1}|$ smaller than $\delta$. To $V_p(X,I)$ one sometimes refers as the \emph{limiting (strong) $p$-variation of $X$ on $I$}, see \cite[Section 4]{TaylorTricot} or \cite{Taylor}. It should not be confused with the $p$-variation, which in general is larger, \cite[p. 4]{LLC}. If $p\geq 1$ and $X$ is H\"older continuous of order $\frac{1}{p}$ on $I$, then obviously $V_p(X,I)$ is finite.

Given $1<p\leq +\infty$, $-\infty<\alpha<n-\frac{n}{p}$ and a subinterval $J$ of $I$ with nonempty interior, we set 
\[\tau_{p,\alpha}(X,J):=\frac{\mathcal{L}^1(J)}{\big\||\xi|^\alpha \hat{\mu}_X^{J}\big\|_{L^p(\mathbb{R}^n)}}.\]
By Remark \ref{R:trivial} we have $\tau_{p,\alpha}(X,J)=0$ if $p<+\infty$ and $\alpha\leq -\frac{n}{p}$ or if $p=+\infty$ and $\alpha<0$. For $1<p<+\infty$ and $-\frac{n}{p}<\alpha<0$ we also consider 
\[\sigma_{p,\alpha}(X,J):=\frac{\mathcal{L}^1(J)}{\big\|\mu_X^{J}\big\|_{\dot{L}_\alpha^{p'}(\mathbb{R}^n)}},\]
where $\frac{1}{p}+\frac{1}{p'}=1$. For $p=2$ we have $\sigma_{2,\alpha}=\tau_{2,\alpha}$, and we can consistently define 
$\sigma_{2,\alpha}$ by this identity for all $-\infty<\alpha<\frac{n}{2}$.

\begin{remark}\label{R:Hausdorffdim}
If $\sigma_{p,\alpha}(X,J)>0$, then $\mu_X^J$ is diffuse enough to have finite $(-\alpha,p')$-energy. In this case $\dim_H X(J)\geq \alpha p+n$, where $\dim_H$ denotes the Hausdorff dimension, as can be seen using standard arguments, \cite[Theorem 8.5]{Mattila}. Note that $\sigma_{p,\alpha}(X,J)\leq \dot{C}_{-\alpha,p}(X(J))$ by (\ref{E:dualdefcap}).
\end{remark}

The following are variants of statements from \cite{Berman69}.
\begin{corollary}\label{C:restrict}
Let $p$ and $\alpha$ be as in (\ref{E:alpharange}) and $q>0$.
\begin{enumerate}
\item[(i)] If $I_1,...,I_M$ are subintervals of $I$ with nonempty interior, then 
\begin{equation}\label{E:prevariation}
\sum_{k=1}^M \diam(X(I_k))^{(\alpha+\frac{n}{p})q}\geq K^q\:\sum_{k=1}^{M} \tau_{p,\alpha}(X,I_k)^q \geq \frac{K^qM^2}{\sum_{k=1}^M \tau_{p,\alpha}(X,I_k)^{-q}}.
\end{equation}
\item[(ii)] Suppose that $(\delta_m)_m$ is a sequence of numbers $\delta_m>0$ converging to zero as $m\to \infty$ and that $I_{m,1},...,I_{m,M_m}$ are subintervals of $I$ of length $0<\mathcal{L}^1(I_i)\leq \delta_m$ and with pairwise disjoint interior. If 
\[\liminf_{m\to \infty} M_m^{-2} \sum_{k=1}^{M_m} \tau_{p,\alpha}(X,I_{m,k})^{-q} =0,\]
then $X$ has infinite limiting $(\alpha +\frac{n}{p})q$-variation $V_{(\alpha +\frac{n}{p})q}(X,I)=+\infty$ and in particular, cannot be H\"older continuous on $I$ of order $((\alpha +\frac{n}{p})q)^{-1}$.
\end{enumerate}
If $1<p<+\infty$ and $-\frac{n}{p}<\alpha<0$, then $\tau_{p,\alpha}(X,\cdot)$ can be replaced by $\sigma_{p,\alpha}(X,\cdot)$.
\end{corollary}

\begin{proof}
The left inequality in (\ref{E:prevariation}) is clear from (\ref{E:Bermangeneral}), the right is just the domination of the harmonic mean by the arithmetic mean. Statement (ii) follows from (i) since
\[\sup_{\Delta(\mathcal{P})\leq 2\delta_m}\sum_i |X(t_i)-X(t_{i-1})|^{(\alpha+\frac{n}{p})q}\geq \sum_{k=1}^{M_m} \diam(X(I_{m,k}))^{(\alpha+\frac{n}{p})q}.\]
The last statement is a consequence of (\ref{E:HY}).
\end{proof}

\begin{remark}\mbox{}
\begin{enumerate}
\item[(i)] For $n=1$, $p=q=2$, $\alpha\geq 0$ and with $I=[0,1]$, $M_m=2^m$ and $I_{m,k}:=[(k-1)2^{-m},k2^{-m}]$, $m\geq 1$, $k=1,...,2^m$, Corollary \ref{C:restrict} had been shown in \cite[(4.2) and Lemma 4.1]{Berman69}.  See also \cite[(22.8) Theorem]{GemanHorowitz}.
\item[(ii)] If $n=1$ and $-\frac{1}{p}<\alpha\leq 0$, then for no $q\geq p$ Corollary \ref{C:restrict} imposes an additional restriction on the range of possible H\"older orders $\gamma$ of $X$. For $n\geq 2$ restrictions on this range can be avoided if $q$ is small enough.
\item[(iii)] Also (\ref{E:prevariation}) is linked with \cite{GG20b}: Suppose that $0<\gamma< 1$, $0<\varrho<n$ and $X$ is $(\gamma,\varrho)$-irregular. Then $\tau_{\infty,\varrho}(X,J)=c^{-1}\mathcal{L}^1(J)^{1-\gamma}$ for all subintervals $J$ of $I$ with nonempty interior, where $c$ is as in Remark \ref{R:Bermangeneral} (iii). For $p=+\infty$, $\alpha=\varrho$ and $q=\frac{1}{1-\gamma}$ the right inequality in (\ref{E:prevariation}) implies that whenever the $I_1,...,I_M$ form a partition of $I$, we have $\sum_{k=1}^{M} \diam(X(I_{k}))^{\varrho/(1-\gamma)}\geq c^{-1/(1-\gamma)}\mathcal{L}^1(I)$, and consequently $V_r(X,I)=+\infty$ for any $r<\frac{\varrho}{1-\gamma}$.
\end{enumerate}
\end{remark}

\begin{remark}
If $X$ satisfies a local H\"older condition $|X(t_0)-X(t)|\leq c_H(t_0)|t_0-t|^\gamma$ for all $t$ from an open interval around a point $t_0\in I$, then for any sufficiently small open interval $J$ containing $t_0$ we must have 
\begin{equation}\label{E:musthave}
c_H(t_0)^{\alpha p+n}\geq \frac{K^p \mathcal{L}^1(J)^{p-\gamma(\alpha p+n)}}{\left\||\xi|^\alpha \hat{\mu}_X^J\right\|_{L^p(\mathbb{R}^n)}^p}.
\end{equation}
If $\lim_{m\to \infty} \big\||\xi|^\alpha \hat{\mu}_X^{J_m}\big\|_{L^p(\mathbb{R}^n)}^p=0$ along a sequence of such intervals $J_m$ with length decreasing to zero, then (\ref{E:musthave}) can hold only if $\gamma<\frac{p}{\alpha p+n}$. For the absolutely continuous case $\alpha\geq 0$ and $n=1$, $p=2$ this absence of local H\"older conditions of order $\frac{2}{2\alpha+n}$ had been shown in \cite[Lemma 4.3]{Berman69}. Condition (\ref{E:musthave}) has the character of a \emph{pointwise density bound}; see \cite[Sections 9-11]{GemanHorowitz} for an earlier discussion of this fact and  \cite[Theorem 63 and Corollary 65]{GG20b} for a modern formulation.
\end{remark}

If we insist that $I_1,...,I_m$ should have pairwise disjoint interiors, then the right hand side in (\ref{E:prevariation}) becomes largest if $I$ is `optimally packed'. To further illustrate this aspect of Corollary \ref{C:restrict}, we briefly comment on a possible view upon (\ref{E:prevariation}) from a \emph{packing measure} perspective, \cite[Section 5.10]{Mattila}, \cite{TaylorTricot}. Let $1<p\leq  +\infty$, $-\infty<\alpha<n-\frac{n}{p}$ and $q>0$. Given $\delta>0$ we define 
\begin{multline}
\mathcal{P}_{p,\alpha,q;\delta}(X,E):=\sup\Big\lbrace \sum_{k=1}^\infty \tau_{p,\alpha}(X,I_k)^q:\ \text{$\{I_k\}_{k=1}^\infty$ family of disjoint closed intervals}\notag\\
\text{$I_k\subset I$ with centers $x_k\in E$ and $0<\mathcal{L}^1(I_k)\leq \delta$}\Big\rbrace
\end{multline}
for any $E\subset I$; we use the agreement that $\sup\emptyset=0$. Since $\delta\mapsto \mathcal{P}_{p,\alpha,q;\delta}(X,E)$ decreases as $\delta\downarrow 0$, the limit $\mathcal{P}_{p,\alpha,q;0}(X,E):=\lim_{\delta\to 0} \mathcal{P}_{p,\alpha,q;\delta}(X,E)$ exists for all $E\subset I$. Standard proofs show that $\mathcal{P}_{p,\alpha,q;0}(X,\cdot)$ is monotone, finitely subadditive, and additive for two sets with positive distance from each other, \cite{Mattila}, \cite{TaylorTricot}. We obviously have $\mathcal{P}_{p,\alpha,q;0}(X,\emptyset)=0$. Setting 
\[\mathcal{P}_{p,\alpha,q}(X,E):=\inf\Big\lbrace\sum_{i=1}^\infty \mathcal{P}_{p,\alpha,q;0}(X,E_i):\ E\subset \bigcup_{i=1}^\infty E_i\Big\rbrace\]
for any $E\subset I$, we obtain a Borel measure $\mathcal{P}_{p,\alpha,q}(X,\cdot)$ on $I$, \cite[Section 4.1]{Mattila}, \cite{TaylorTricot}. 
\begin{remark}
One could similarly define measures based on $\sigma_{p,\alpha}(X,\cdot)$. Since $\sigma_{p,\alpha}(X,J)$ is the reciprocal of an energy, this might even be more natural for certain purposes.
\end{remark}

The following could be seen as a reformulation of  Corollary \ref{C:restrict} (i) in terms of the measures $\mathcal{P}_{p,\alpha,q}(X,\cdot)$.
\begin{corollary}\label{C:restrictviameasure}
Let $1<p<+\infty$, $-\frac{n}{p}<\alpha<0$ and $q>0$. Then for any closed subinterval $J$ of $I$ we have 
$\mathcal{P}_{p,\alpha,q}(X,J)\leq K^{-q}\:V_{(\alpha+\frac{n}{p})q}(X,J)$.
\end{corollary}

\begin{proof}
We may assume that $V_{(\alpha+\frac{n}{p})q}(X,J)<+\infty$. Let $\varepsilon>0$. Choose $\delta_\varepsilon>0$ small enough to have 
\[S(\delta):=\sup_{\Delta(\mathcal{P}\leq 2\delta}\sum_{i=1}^M|X(t_i)-X(t_{i-1})|^{(\alpha+\frac{n}{p})q}\leq V_{(\alpha+\frac{n}{p})q}(X,J)+\varepsilon\]
for all $0<\delta<\delta_\varepsilon$. Now suppose that $\{I_k\}_{k=1}^\infty$ is a family of disjoint closed intervals $I_k\subset I$ with centers in $J$ and $0<\mathcal{L}^1(I_k)\leq \delta$. Then for all finite $M$ we have we have 
\[\sum_{k=1}^M \sigma_{p,\alpha}(X,I_k)^q\leq K^{-q}\sum_{k=1}^M \diam(X(I_k))^{(\alpha+\frac{n}{p})q}\leq K^{-q}S(\delta)\]
by Corollary \ref{C:restrict} (i), and since the right hand side does not depend on $M$, we can replace $M$ by $\infty$. Taking the supremum over all such $\{I_k\}_{k=1}^\infty$, we see that $\mathcal{P}_{p,\alpha,q}(X,J)\leq \mathcal{P}_{p,\alpha,q;0}(X,J)\leq \mathcal{P}_{p,\alpha,q;\delta}(X,J)\leq K^{-q}(V_{(\alpha+\frac{n}{p})q}(X,J)+\varepsilon)$.
\end{proof}

For continuous $X$, $p=2$ and fixed $q>0$ we can observe a well-defined discontinuity with respect to $\alpha$. Since $I$ is a bounded interval, a continuous function $X:I\to \mathbb{R}^n$ is absolutely continuous on $I$ and therefore admits a modulus of continuity $\omega_X$, more precisely, there is a nondecreasing function $\omega_X:[0,+\infty)\to [0,+\infty)$ such that $\omega(0+)=\omega(0)=0$ and we have $|X(t)-X(s)|\leq \omega_X(|t-s|)$ for all $s,t\in I$.

\begin{lemma}\label{L:jump}
Let $-\infty<\beta<\alpha<0$ and $q>0$. Suppose that $X$ is continuous on $I$ and let $\omega_X$ be a modulus of continuity for $X$. 
\begin{enumerate}
\item[(i)] We have $\sigma_{2,\beta}(X,J)\leq c\:\omega^{\alpha-\beta}\sigma_{2,\alpha}(X,J)$ for any nonempty open subinterval $J$ of $I$; here $c>0$ is a constant depending only on $n$, $\alpha$ and $\beta$. 
\item[(ii)] If $E\subset I$ is such that $\mathcal{P}_{2,\alpha,q}(X,E)<+\infty$, then $\mathcal{P}_{2,\beta,q}(X,E)=0$.
\end{enumerate} 
\end{lemma}
\begin{proof}
Statement (i) it suffices to note that $\diam X(J)\leq \omega_X(\mathcal{L}^1(J))$ and therefore 
\[\int_J\int_J |X(t)-X(s)|^{-2\alpha-n}ds\:dt\leq c^2 \omega_X(\mathcal{L}^1(J))^{2(\beta-\alpha)}\int_J\int_J |X(t)-X(s)|^{-2\beta-n}ds\:dt\]
with $c>0$ depending only on $\alpha$, $\beta$ and $n$. Statement (ii) follows from (i) by the same standard arguments as used for packing measures.
\end{proof}

Lemma \ref{L:jump} implies that for continuous $X$, $q>0$ and any $E\subset I$ the number 
\begin{multline}
\ind_q(X,E):=\frac{n}{2}+\sup\left\lbrace -\infty<\alpha < 0:\ \mathcal{P}_{2,\alpha,q}(X,E)=0\right\rbrace\notag\\
=\frac{n}{2}+\inf\left\lbrace -\infty<\alpha < 0:\ \mathcal{P}_{2,\alpha,q}(X,E)=+\infty\right\rbrace
\end{multline}
is well-defined, and it is an element of $[0,\frac{n}{2}]$. One could call it the \emph{occupation index of $X$ over $E$}. The following is  immediate from Corollary \ref{C:restrictviameasure}.
\begin{corollary}
Let $q>0$ and suppose that $X$ is continuous on $I$. If $-\frac{n}{2}<\alpha<0$ and $J$ is a closed subinterval of $I$ such that $V_{(\alpha+\frac{n}{2})q}(X,J)<+\infty$, then $\ind_q(X,J)\geq \alpha+\frac{n}{2}$.
\end{corollary}

\begin{remark}
In view of \cite[Theorem 4.1]{TaylorTricot} it would be interesting to see whether one can relate $\mathcal{P}_{2,\alpha,q}(X,E)$ and $\ind_q(X,E)$ with the packing measure and packing dimension of the image $X(E)$ of $E\subset I$ under $X$.
\end{remark}


\end{document}